\documentclass[12pt]{amsart}
\usepackage[margin=0.5in]{geometry}
\usepackage{graphicx,epstopdf,color}
\usepackage{amssymb,amscd,amsthm,amsxtra,mathtools}
\usepackage[alphabetic]{amsrefs}
\usepackage[font=footnotesize]{caption}

\usepackage{mathrsfs}
\usepackage{enumerate}
\renewcommand{\mathcal}{\mathscr}

\usepackage{hyperref}

\numberwithin{equation}{section}

\newtheorem{theorem}{Theorem}[section]
\newtheorem{lemma}[theorem]{Lemma}

\newtheorem{corollary}[theorem]{Corollary}

\newcommand{\R}{\mathbb R}
\newcommand{\N}{\mathbb N}

\renewcommand{\leq}{\leqslant}
\renewcommand{\le}{\leqslant}
\renewcommand{\geq}{\geqslant}
\renewcommand{\ge}{\geqslant}
\renewcommand{\epsilon}{\varepsilon}

\newcommand{\e}{\varepsilon}

\author{Serena Dipierro}
\author{Veronica Felli}
\author{Enrico Valdinoci}

\title[Unique continuation principle]{Unique continuation principles
in cones under nonzero Neumann boundary conditions}\thanks{The authors are member of INdAM/GNAMPA.
S.~Dipierro is supported by
the Australian Research Council DECRA DE180100957
``PDEs, free boundaries and applications'' and the Fulbright Foundation.
S.~Dipierro and E.~Valdinoci are supported by the Australian Research Council
Discovery Project DP170104880 NEW ``Nonlocal Equations at Work''.
V.~Felli is partially supported by the PRIN-2015
grant ``Variational methods, with applications to problems in
mathematical physics and geometry''. This work was started on the occasion
of a very fruitful visit of V.~Felli to the University of Melbourne.}

\begin{document}

\maketitle

\begin{abstract}
We consider an elliptic equation in a cone, endowed
with (possibly inhomogeneous) Neumann conditions.
The operator and the forcing terms can also
allow non-Lipschitz singularities at the vertex of the cone.

In this setting, we provide unique continuation results, both
in terms of interior and boundary points.

The proof relies on a suitable Almgren-type frequency formula
with remainders. As a byproduct, we obtain classification results
for blow-up limits.
\end{abstract}

\section{Introduction}

In this article we consider an elliptic equation
with Neumann boundary condition. The domain taken
into consideration is a cone, and the equation and the boundary condition
can be inhomogeneous and be singular at the origin.

The main results that we provide are of unique continuation type.
Roughly speaking, we will show that {\em if a solution vanishes at any order
at the vertex of the cone, then the solution must necessarily vanish
in a neighborhood of the vertex} (and then everywhere, up to suitable assumptions).

The notion of vanishing can be framed both with respect to the convergence
of points coming from the interior of the domain and, under the appropriate
assumptions, with respect to the convergence
of points coming from the boundary.

{F}rom these results, we also obtain classification results for the blow-up limits.
The method of proof will rely on the special geometric structure
of the cone, which is a set invariant under dilations
and in which the normal on the side of the cone is perpendicular to the
radial direction. The main analytic tool in use will be an appropriate
type of frequency function. Differently from the classical case in~\cite{MR574247},
the choice of the frequency function in our case has to comprise additional quantities
and reminders to deal with the forcing terms and possibly compensate for the singular
behaviors near the vertex.
\medskip

The mathematical setting in which we work is the following.
We let~$\Omega
\subseteq\R^n$, with~$n\ge2$, be a cone with vertex at the
origin (namely, we assume that~$x\in\Omega$
if and only if~$tx\in\Omega$ for all~$t>0$).
We consider the spherical cap
\begin{equation}\label{SPHCAP}
\Sigma=\left\{\frac{x}{|x|}:x\in\Omega\right\}\subset{\mathbb S}^{n-1}
\end{equation}
and we assume that $\Sigma$ has~$C^2$ boundary in~${\mathbb S}^{n-1}$.

We also take into account a positive function~$A\in W^{1,1}(\Omega)$
such that 
\begin{equation}\label{MUCK}
c\leq A(x)\leq \frac1c\quad {\mbox{ for some $c>0$ and a.e. $x\in\Omega$.}}
\end{equation}
For every $r>0$ we denote $B_r=\{x\in\R^n:|x|<r\}$.
 We deal with weak solutions of the following partial differential
 equation in a neighbourhood of the vertex of the cone (to fix the
 notations we consider $\Omega\cap B_1$) 
with possibly inhomogeneous Neumann datum:
\begin{equation}\label{MAIN EQ}
\begin{cases}
{\rm div}\,\big(A(x)\,\nabla u(x)\big)=g(x,u(x)), & {\mbox{ for every
  }}x\in\Omega\cap B_1,\\ 
A(x)\nabla u(x)\cdot \nu(x)=f(x,u(x)),&
{\mbox{ for every }}x\in B_1\cap \partial\Omega,
\end{cases}
\end{equation}
where~$\nu(x)$ denotes the exterior unit normal of~$\Omega$ at~$x\in\partial\Omega$,
$f\in C^1((\overline{\Omega}\setminus\{0\})\times \R)$, 
and $g:\Omega\times \R\rightarrow \R$ is a 
Carath\'eodory function.

We say that a function $u\in H^1(B_1\cap\Omega)$ is a weak solution to 
\eqref{MAIN EQ} if, for all $\varphi\in C^\infty_{\rm c}(B_1\cap \overline{\Omega})$,
\begin{equation}\label{eq:weak_solution}
\int_{B_1\cap\Omega} A(x)\,\nabla u(x)\cdot\nabla \varphi(x)\,dx=-\int_{B_1\cap\Omega}
g(x,u(x))\varphi(x)\,dx+\int_{B_1\cap\partial\Omega}f(x,u(x))\varphi(x)
\,d{\mathcal{H}}^{n-1}_x.
\end{equation}
As a technical observation, we
point out that the integrals at the right hand side of the above
identity are finite under the assumptions of Theorem~\ref{MONOTONIA} below  in view of the Poincar\'e-type Inequality and
the Trace Inequality proved in Corollary~\ref{1.3} and Lemma~\ref{TRACEIN1}  respectively. 

The use of Almgren-type frequency functions to study unique
continuation properties of elliptic partial differential equations
dates back to the pioneering contribution of  Garofalo and Lin
\cite{MR833393} and relies essentially on the possibility of deducing
from the boundedness of the frequency quotient a doubling-type
condition.
Unique continuation from boundary points was investigated via
Almgren-type monotonicity arguments  in \cite{MR1466583,MR1363203,MR3109767,MR1415331,MR2370633}.
As far as elliptic equations with Neumann-type boundary conditions
are concerned, we mention that in \cite{MR2162295}  boundary unique
continuation theorems and 
 doubling properties near the boundary were established under zero Neumann boundary conditions.
The main novelty of the present paper is a strong unique
continuation result for solutions whose restriction to the boundary vanishes at any order at the vertex
under non-homogeneous Neumann boundary conditions, while
in \cite[Theorem 1.7]{MR2162295} unique
continuation from the boundary was proved for solutions vanishing on
positive surface measure subsets of the boundary and satisfying a zero
 Neumann condition on such set.
The achievement of such a
result requires a combination of the
monotonicity argument with a blow-up analysis for scaled solutions, in
the spirit of  \cite{MR2735078,MR3169789}.

We now introduce the notation needed to define the frequency function  for our setting.
For~$r>0$, we define
\begin{equation}\label{DeH DEF}
\begin{split}
D(r)\,&:= r^{2-n}\int_{B_r\cap\Omega} A(x)\,|\nabla u(x)|^2\,dx
-r^{2-n}\int_{B_r\cap\partial\Omega} f(x,u(x))\,u(x)\,d{\mathcal{H}}^{n-1}_x
\\&\qquad\qquad+r^{2-n}\int_{B_r\cap \Omega} g(x,u(x))\,u(x)\,dx
\\
{\mbox{and }}\qquad H(r)\,:&=r^{1-n}\int_{\partial B_r\cap\Omega} A(x)\,u^2(x)
\,d{\mathcal{H}}^{n-1}_x\\&=
\int_{\Sigma} A(ry)\,u^2(ry)
\,d{\mathcal{H}}^{n-1}_y.
\end{split}
\end{equation}
We also introduce the ``Almgren frequency function'' in our framework, given by
\begin{equation}\label{DeH DEN}
{\mathcal{N}}(r):=\frac{D(r)}{H(r)}.
\end{equation}

With this setting, the pivotal result that we obtain is an appropriate
monotonicity formula with reminders, which we state as follows:

\begin{theorem}\label{MONOTONIA}
Suppose that \eqref{MUCK} holds and 
\begin{alignat}{3}\label{STR:HY1}
& |\nabla A(x)\cdot x|\le \e_r\,A(x),&&\qquad{\mbox{for a.e. $x\in B_r\cap\Omega$,
with }}\lim_{r\searrow0} \e_r=0,\\ 
& \label{TANGE}
|\nabla A(x)|\le \frac{C\,A(x)}{|x|},&&\qquad{\mbox{for a.e. $x\in
    B_1\cap\Omega$,}}\\
&\label{STR:HY2} |f(x,t)|\le C\,A(x)\,|x|^{\delta-1}\,|t|,
&&\qquad{\mbox{for a.e. $x\in \Omega \cap B_1$ and any
    $t\in\R$,}}\\
&\label{STR:HY2DER} |\nabla_x f(x,t)|\le C\,A(x)\,|x|^{\delta-2}\,|t|,
&&\qquad{\mbox{for a.e. $x\in \Omega \cap B_1$ and any $t\in\R$,}}\\
{\mbox{and}}\quad&
\label{g:grow} |g(x,t)|\le C\,A(x)\,|x|^{\delta-2}
                \,|t|,&&\qquad{\mbox{for a.e. $x\in B_1\cap\Omega$
and any $t\in\R$,}}
\end{alignat}
for some~$C>0$ and~$\delta>0$.

Let also
\begin{equation}\label{HYhstar} F(x,t):=\int_0^t f(x,\tau)\,d\tau.\end{equation}
Let
\begin{equation}\label{Linfty}
  u\in H^1(\Omega\cap B_1)\cap 
  L^\infty(\Omega\cap B_1)
\end{equation}
 be a 
solution of~\eqref{MAIN EQ} in the sense of \eqref{eq:weak_solution}, such that
\begin{equation}\label{NONTRIVIAL}
{\mbox{$u\not\equiv0$ in~$\Omega\cap B_{r}$,}}
\end{equation}
for all $r\in(0,1)$.

Then the following holds true.
\begin{enumerate}[\rm (i)]
\item There exists $r_0>0$ such that 
\begin{equation}\label{eq:H-N+1-pos}
H(r)>0\quad\text{and} \quad \mathcal N(r)+1>0\quad\text{for all
  $r\in(0,r_0)$};
\end{equation}
  in particular the function $\mathcal N$ defined in \eqref{DeH DEN}
  is well defined on $(0,r_0)$.
\item There exist $r_1\in (0,r_0)$ and $C_1>0$ such that
  \begin{equation}\label{eq:stima-N'}
        {\mathcal{N}}'(r)\geq
        -C_1\,\max\{r^\delta,\e_r\}r^{-1}(2+\mathcal
        N(r))\quad\text{for all }r\in(0,r_1).
  \end{equation}
\item If also
\begin{equation}\label{L1}
r\mapsto \frac{\e_r}{r}\in L^1(0,r_1),\end{equation} then  the limit
  \begin{equation}\label{LIMITgamma}
\gamma:=\lim_{r\searrow0}\mathcal N(r)
\end{equation}
exists, is finite and $\gamma\geq 0$.
\end{enumerate}
\end{theorem}

We observe that the assumptions of Theorem~\ref{MONOTONIA}
are very general and do not necessarily require the weight~$A$ to be
Lipschitz continuous or the source terms~$f$ and~$g$
to be bounded.
In particular, estimate~\eqref{eq:stima-N'} requires assumptions~\eqref{STR:HY1} and~\eqref{TANGE} which could be
satisfied even by unbounded potentials, as for example $A(x)= \log|x|
( \cos (x_n/|x|) - 2 )$. On the other hand, to prove that $\mathcal N$
is bounded and has finite limit as $r\to0^+$ assumption~\eqref{L1}
is also needed; we observe that~\eqref{L1} forces the
boundedness of~$A$
but could be satisfied by non-Lipschitz continuous weights, like
$A(x)=1+|x|^\delta$ with $\delta$ positive and small, for example.

The functions $f$ and~$g$ can be singular as well,
in accordance with~\eqref{STR:HY2}
and~\eqref{g:grow}. To allow all these possible singularities,
it is crucial that the ``frequency function''
also takes into account the special behaviors of~$A$, $f$
and~$g$, as in~\eqref{DeH DEF}.
Moreover, the special geometry of the cone~$\Omega$
will turn out to be the cornerstone
for our main estimates to hold, thus providing an interesting
interplay between analytic and geometric properties of the problem.

We also observe that condition~\eqref{NONTRIVIAL} is quite natural, since
it requires that 
the solution is nontrivial in any neighborhood of the vertex
of the cone. Furthermore, under the additional assumption that~$A$ is
locally Lipschitz
continuous, assumption~\eqref{NONTRIVIAL} is satisfied
by all nontrivial solutions, in light of the classical unique
continuation principle in~\cite{MR882069}, see
also \cite{MR1233189}
(similarly, if~$A$ satisfies a  Muckenhoupt-type assumption,
then~\eqref{NONTRIVIAL} is a consequence of the unique continuation principle in~\cite{MR2370633}, see also~\cite{MR833393}).\medskip

{F}rom Theorem~\ref{MONOTONIA} and a ``doubling property'' method
one obtains a number of results of unique continuation type. In this
spirit, we first provide a unique continuation result
from the vertex of the cone with respect to interior points:

\begin{theorem}\label{UQ:1}
Let~$u$ be a 
solution of~\eqref{MAIN EQ}, under assumptions~\eqref{MUCK},
\eqref{STR:HY1}, \eqref{TANGE}, \eqref{STR:HY2}, \eqref{STR:HY2DER},
\eqref{g:grow}, \eqref{Linfty} and~\eqref{L1}.

Assume also that~$u$ vanishes at the origin at any order with
respect to interior points, namely that for every~$k\in\N$
\begin{equation}\label{UQ:2}
\lim_{\Omega\ni x\to 0} \frac{u(x)}{|x|^k}=0.
\end{equation}
Then there exists~$r>0$ such that
\begin{equation}\label{UQ:3}
{\mbox{$u\equiv0$ in~$\Omega\cap B_r$.}}\end{equation}
If, in addition,
$A$ is locally Lipschitz
continuous, then
\begin{equation}\label{UQ:4}
{\mbox{$u\equiv0$ in~$\Omega\cap B_1$.}}\end{equation}
\end{theorem}

An interesting consequence of our Theorem \ref{MONOTONIA} deals with
blow-up limits. Namely, for each~$\lambda>0$,
we define
\begin{equation} \label{ulam}
u_\lambda(x):=\frac{u(\lambda x)}{\sqrt{H(\lambda)}}.
\end{equation}
We consider the Laplace-Beltrami
operator 
${\mathcal L}_\Sigma:=-\Delta_{\mathbb S^{n-1}}$ on the spherical cap
$\Sigma$ under null Neumann boundary conditions.
  By classical spectral theory, the spectrum of the  operator
${\mathcal L}_\Sigma$ is discrete and consists in a nondecreasing diverging sequence of
eigenvalues $0=\lambda_1(\Sigma)< \lambda_2(\Sigma)\leq\cdots\leq \lambda_k(\Sigma)\leq\cdots$
with finite multiplicity.

 In the following theorem we describe the
limit profiles of the blowed-up family \eqref{ulam} in terms of the
eigenvalues and the eigenfunctions of $\mathcal L_\Sigma$.

\begin{theorem}\label{BLOW-a}
Let~$u$ be a 
solution of~\eqref{MAIN EQ}, under assumptions~\eqref{MUCK},
\eqref{STR:HY1}, \eqref{TANGE}, \eqref{STR:HY2}, \eqref{STR:HY2DER},
\eqref{g:grow}, \eqref{Linfty} and~\eqref{L1}.

Assume that~\eqref{NONTRIVIAL} holds true,
\begin{equation}\label{eq:f_t}
|f_t(x,t)|\le C\,|x|^{\delta-1},
\quad\mbox{for a.e. $x\in \Omega \cap B_1$ and any $t\in\R$},
\end{equation}
 and that
\begin{equation}\label{ACONT}
\lim_{x\to0} A(x)=1.
\end{equation}
Then, up to a subsequence,
as~$\lambda\searrow0$, we have that~$u_\lambda$
converges strongly in~$H^1(\Omega\cap B_1)$ to a function~$\tilde u$ which
is positively homogeneous and can be written in the form
\begin{equation}\label{tildeupsi} \tilde u(x)= |x|^\gamma\,\psi\left( \frac{x}{|x|}\right),\end{equation}
where
\[
\gamma=-\frac{n-2}2+\sqrt{\left(\frac{n-2}2\right)^2+\lambda_{k_0}(\Sigma)}
\ge0
\]
for some $k_0\in\N\setminus\{0\}$
 and~$\psi$ is an eigenfunction of the operator $\mathcal L_\Sigma$
 associated to the eigenvalue $\lambda_{k_0}(\Sigma)$ such that
\begin{equation}\label{NORMALIZ}
\int_{\Sigma}
\psi^2(x)\,d{\mathcal{H}}^{n-1}_x=1.
\end{equation}
\end{theorem}

{F}rom Theorem~\ref{BLOW-a}, one can also
obtain
a unique continuation result
from the vertex of the cone with respect to boundary points:

\begin{theorem}\label{UnA-ass}
Let~$u$ be a 
solution of~\eqref{MAIN EQ}, under assumptions~\eqref{MUCK},
\eqref{STR:HY1}, \eqref{TANGE}, \eqref{STR:HY2}, \eqref{STR:HY2DER},
\eqref{g:grow}, \eqref{Linfty}, \eqref{L1}, \eqref{eq:f_t} and~\eqref{ACONT}.

Assume also that~$u$ vanishes at the origin at any order with
respect to boundary points, namely that for every~$k\in\N$
\begin{equation}\label{UQ:2X}
\lim_{\partial\Omega\ni x\to 0} \frac{u(x)}{|x|^k}=0.
\end{equation}
Then there exists~$r>0$ such that
\begin{equation}\label{UQ:3X}
{\mbox{$u\equiv0$ in~$\Omega\cap B_r$.}}\end{equation}
If, in addition,
$A$ is locally Lipschitz
continuous, then
\begin{equation}\label{UQ:4X}
{\mbox{$u\equiv0$ in~$\Omega$.}}\end{equation}
\end{theorem}

We stress that while~\eqref{UQ:2} is assumed for interior points,
we have that hypothesis~\eqref{UQ:2X} focuses on boundary points.\medskip

The rest of the article is organized as follows.
Section~\ref{TTOL} presents a number of ancillary results, to be exploited
in the proofs of the main theorems. In particular, we will collect there
some observations on the geometry of the cone and suitable functional inequalities.

The proof of Theorem~\ref{MONOTONIA} is presented in Section~\ref{Orafgytise}
and will serve as a pivotal result for the main theorems of this paper.
Namely, 
Theorem~\ref{UQ:1} will be proved in Section~\ref{9SIoryuitis},
Theorem~\ref{BLOW-a} will be proved in Section~\ref{senrtoesyuiert03},
and Theorem~\ref{UnA-ass}
will be proved in Section~\ref{oedkcpoetrtegfi}.

\section{Toolbox}\label{TTOL}

This section collects ancillary results used in the main proofs.

\subsection{Cone structure}

We recall here an elementary property of the cones:

\begin{lemma}
Let~$\Omega\subset\R^n$ be a cone with respect to the origin. Then
\begin{equation}\label{0102303}
\nu(x)\cdot x=0\qquad{\mbox{for any }}x\in\partial\Omega\setminus\{0\}.
\end{equation}\end{lemma}

\begin{proof}
Fixed~$x_0\in\partial\Omega\setminus\{0\}$, we have that there
exists~$r_0>0$ such that~$\Omega\cap B_r(x_0)$
coincides with the sublevel sets of some nondegenerate
function~$\Phi_0:B_r(x_0)\to\R$,
with~$\nu(x)=\frac{\nabla\Phi_0(x)}{|\nabla\Phi_0(x)|}$.
By the cone structure of~$\Omega$, we thereby see that, for any~$t$ close to~$1$,
$$ 0=\Phi_0(x_0)=\Phi_0(t x_0),$$
and so
$$ 0=\left.\frac{d}{dt} \Phi_0(t x_0)\right|_{t=1}=\nabla \Phi_0(x_0)\cdot x_0=
|\nabla\Phi(x_0)|\,\nu(x_0)\cdot x_0.$$
This proves
that~$\nu(x_0)\cdot x_0=0$ and establishes~\eqref{0102303}.
\end{proof}

\subsection{A Poincar\'e-type Inequality}

In this subsection, we provide some results concerning suitable
weighted Poincar\'e-type Inequalities which will play an important role
in some of the technical estimates needed to prove the main results.


\begin{lemma}\label{POIV} Let~$\mu\in(-\infty,n)$.
Let~$\Omega\subset\R^n$ be a cone with respect to the origin such that
the spherical cap $\Sigma$ defined in \eqref{SPHCAP} is smooth.
Let $A\in \L^\infty(\Omega)$ satisfy \eqref{TANGE}.
For every $r>0$ and $u\in C^\infty(\overline{\Omega\cap B_r})$  
\[
\int_{\Omega\cap B_r} \left(\frac {n-\mu}2 A(x)+\nabla A(x)\cdot
  x\right)\frac{u^2(x)}{|x|^\mu}\,dx
\leq \frac1{r^{\mu-1}}\int_{\partial B_r\cap
  \Omega}A(x)u^2(x)\,d{\mathcal{H}}^{n-1}_x+\frac{2}{n-\mu}\int_{\Omega\cap B_r}\frac{
  A(x)|\nabla u(x)|^2}{|x|^{\mu-2}}\,dx.
\]
\end{lemma}

\begin{proof} Let $u\in C^\infty(\overline{\Omega\cap B_r})$. Since 
\[
\mathop{\rm div}\left(Au^2\frac{x}{|x|^\mu}\right)=\frac{n-\mu}{|x|^\mu}Au^2
+u^2\nabla A \cdot \frac{x}{|x|^\mu}+2Au\nabla u\cdot \frac{x}{|x|^\mu},
\]
by the Divergence Theorem and \eqref{0102303} we deduce that 
\begin{align*}&
(n-\mu)\int_{\Omega\cap B_r} \frac{A(x)u^2(x)}{|x|^\mu}\,dx\\&=
\int_{\Omega\cap B_r} \left[
\mathop{\rm div}\left(A(x)u^2(x)\frac{x}{|x|^\mu}\right)
-u^2\nabla A (x)\cdot \frac{x}{|x|^\mu}-2A(x)\,u(x)\nabla u(x)\cdot \frac{x}{|x|^\mu},
\right]
\\&= \frac1{r^{\mu-1}}\int_{\partial B_r\cap
  \Omega}A(x)\,u^2(x)\,d{\mathcal{H}}^{n-1}-\int_{\Omega\cap B_r}
  \frac{(\nabla A(x)\cdot x) u^2(x)}{|x|^\mu}\,dx
-2\int_{\Omega\cap B_r}\frac{A(x)u\nabla u(x)\cdot x}{|x|^\mu}\,dx\\
&\leq \frac1{r^{\mu-1}}\int_{\partial B_r\cap
  \Omega}A(x)u^2(x)\,d{\mathcal{H}}^{n-1}-\int_{\Omega\cap B_r}\frac{
  (\nabla A(x)\cdot x) u^2(x)}{|x|^\mu}\,dx
+\frac {n-\mu}2\int_{\Omega\cap B_r}\frac{A(x)u^2(x)}{|x|^\mu}\,dx\\
&\quad+\frac{2}{n-\mu}\int_{\Omega\cap B_r}
\frac{A(x)|\nabla u(x)|^2}{|x|^{\mu-2}}\,dx,
\end{align*}
and hence the conclusion follows.
\end{proof}

\begin{corollary}\label{1.3} Let~$\mu\in(-\infty,n)$.
Let~$\Omega\subset\R^n$ be a cone with respect to the origin such that
the spherical cap $\Sigma$ defined in \eqref{SPHCAP} is smooth.
Let $c\in\left(0,\frac{n-\mu}{2}\right)$ and $A\in \L^\infty(\Omega)$
satisfy \eqref{TANGE}
and \eqref{STR:HY1}.
Then there exists~$r_\mu>0$ such that
for every $r\in(0,r_\mu)$ and $u\in C^\infty(\overline{\Omega\cap B_r})$  
$$
c\int_{\Omega\cap B_r} \frac{A(x)u^2(x)}{|x|^{\mu}}\,dx
\leq \frac1{r^{\mu-1}}\int_{\partial B_r\cap
  \Omega}A(x)u^2(x)\,d{\mathcal{H}}^{n-1}_x+\frac{2}{n-\mu}
  \int_{\Omega\cap B_r}\frac{A(x)|\nabla u(x)|^2}{|x|^{\mu-2}}\,dx.$$
\end{corollary}

\begin{proof} Exploiting~\eqref{STR:HY1}, we observe that
$$\frac {n-\mu}2 A(x)+\nabla A(x)\cdot x\ge
\frac {n-\mu}2 A(x)-
\e_r\,A(x)\ge cA(x),
$$
as long as~$r$ is small enough, and hence the desired result follows by Lemma~\ref{POIV}.
\end{proof}

For $\mu<2$ the previous corollary yields the following result.
\begin{corollary}\label{cor:dis} 
Let $\mu<2$. Let~$\Omega\subset\R^n$ be a cone with respect to the origin such that
the spherical cap $\Sigma$ defined in \eqref{SPHCAP} is smooth.
Let $c\in\left(0,\frac{n-\mu}{2}\right)$ and $A\in \L^\infty(\Omega)$
satisfy \eqref{TANGE}
and \eqref{STR:HY1}.
Then there exists~$r_\mu>0$ such that,
for every $r\in(0,r_\mu)$ and $u\in H^1(\Omega\cap B_r)$,
$u|x|^{-\mu/2}\in L^2(\Omega\cap B_r)$ and 
$$
c\int_{\Omega\cap B_r} \frac{A(x)u^2(x)}{|x|^{\mu}}\,dx
\leq \frac1{r^{\mu-1}}\int_{\partial B_r\cap
  \Omega}A(x)u^2(x)\,d{\mathcal{H}}^{n-1}_x+\frac{2\, r^{2-\mu}}{n-\mu}
  \int_{\Omega\cap B_r}A(x)|\nabla u(x)|^2\,dx.$$
\end{corollary}
\begin{proof}
  The inequality for $u\in C^\infty(\overline{\Omega\cap B_r})$
  follows esily from Corollary \ref{1.3} and the fact that, since
  $2-\mu>0$, $|x|^{2-\mu}\leq r^{2-\mu}$ in $\Omega\cap B_r$. The
  conclusion follows by density and the Fatou's Lemma. 
\end{proof}

\subsection{Trace Inequalities}

Now we present a result of trace-type which will be exploited in the proofs
of the main theorems.

\begin{lemma}\label{TRACEIN1} Let~$\gamma \in(-\infty,n-1)$.
Let~$\Omega\subset\R^n$ be a cone with respect to the origin such that
the spherical cap $\Sigma$ defined in \eqref{SPHCAP} is smooth.
Let $A\in \L^\infty(\Omega)$
satisfy \eqref{MUCK}. For every $r>0$ and $u\in C^\infty(\overline{\Omega\cap B_r})$  we have that
$$ \int_{\partial \Omega\cap B_r}\frac{A(x)\,u^2(x) }{|x|^\gamma}
\,d{\mathcal{H}}^{n-1}\le C\,\int_{\Omega\cap B_r}\left[
\frac{A(x)\,|\nabla u(x)|^2}{|x|^{\gamma-1}}
+\frac{A(x) u^2(x)}{|x|^{\gamma+1}}\right]\,dx,$$
for some~$C>0$ independent of $r$.
Furthermore, if $\gamma<1$, then every function   
$u\in H^1(\Omega\cap B_r)$ has a trace belonging to $L^2(\partial
\Omega\cap B_r;|x|^{-\gamma/2})$ and  
$$ \int_{\partial \Omega\cap B_r}\frac{A(x)\,u^2(x) }{|x|^\gamma}
\,d{\mathcal{H}}^{n-1}\le C\,
\left[
r^{1-\gamma}\int_{\Omega\cap B_r}
A(x)\,|\nabla u(x)|^2\,dx+\int_{\Omega\cap B_r}
\frac{A(x) u^2(x)}{|x|^{\gamma+1}}\right]\,dx.$$
\end{lemma}

\begin{proof} We let~$u\in C^\infty(\overline{\Omega\cap B_r})$.
Also, for all~$\rho\in(0,r)$ and~$\theta\in\Sigma$,
we define~$u^{(\rho)}(\theta):=u(\rho\theta)$.
By Fubini's Theorem and the
Sobolev Trace Theorem on manifolds we have that 
\begin{equation*}\begin{split}
\int_{\partial \Omega\cap B_r}\frac{\,u^2(x) }{|x|^\gamma}
\,d{\mathcal{H}}^{n-1}&=
\int_0^r\rho^{-\gamma}\left(\int_{\partial \Omega\cap \partial
  B_\rho}
u^2
\,d{\mathcal{H}}^{n-2}\right)\,d\rho\\
&=\int_0^r\rho^{-\gamma+n-2}\left(\int_{\partial \Sigma}
u^2(\rho\theta)
\,d\theta\right)\,d\rho\\
&=\int_0^r\rho^{-\gamma+n-2}\left(\int_{\partial \Sigma}|
u^{(\rho)}(\theta)|^2
\,d\theta\right)\,d\rho\\
&\leq C\,\int_0^r\rho^{-\gamma+n-2}\left(\int_{\Sigma}
\Big(|\nabla_\theta u^{(\rho)}(\theta)|^2+|u^{(\rho)}(\theta)|^2\Big)
\,d\theta\right)\,d\rho,
\end{split}\end{equation*}
where~$\nabla_\theta$ denotes the tangential gradient along~$\Sigma$,
so that, if~$x=\rho\theta$,
\begin{eqnarray*}
|\nabla_\theta u^{(\rho)}(\theta)|=
\rho\left|
\nabla u(x)-\big(\nabla u(x)\cdot x\big)\frac{x}{|x|^2}
\right|\le \rho\,|\nabla u(x)|.
\end{eqnarray*}
Hence, in view of \eqref{MUCK}, we find that
\begin{eqnarray*}
\int_{\partial \Omega\cap B_r}\frac{A(x)\,u^2(x) }{|x|^\gamma}
\,d{\mathcal{H}}^{n-1}
&\le&
\frac Cc\,\int_0^r\rho^{-\gamma+n-2}\left(\int_{\Sigma}
\Big(
\rho^2 |\nabla u (\rho\theta)|^2+u^2(\rho\theta)\Big)
\,d\theta\right)\,d\rho
\\&=&
\frac Cc\,\int_{\Omega\cap B_r}|x|^{-\gamma-1}
\Big(
|x|^2 |\nabla u(x)|^2+ u^2(x)\Big)
\,dx \\&\leq&
\frac C{c^2}\,\int_{\Omega\cap B_r}|x|^{-\gamma-1}
\Big(
|x|^2 A(x)|\nabla u(x)|^2+A(x) u^2(x)\Big)
\,dx,
\end{eqnarray*}
which yields the inequality for functions in
$C^\infty(\overline{\Omega\cap B_r})$. If $\gamma<1$, then
$|x|^{1-\gamma}\leq r^{1-\gamma}$ in $\Omega\cap B_r$, then  The
  conclusion follows by density and the Fatou's Lemma.
\end{proof}

\section{Proof of Theorem~\ref{MONOTONIA}}\label{Orafgytise}

We first observe that,  by 
 elliptic regularity theory (see e.g. Theorem~8.13
 in~\cite{MR2399851}, \cite{MR0125307, MR0162050} or
 \cite{MR0350177}) we have that, under the assumptions of Theorem
 \ref{MONOTONIA}, 
\begin{equation}\label{eq:regH2}
u\in H^2\big(\Omega\cap(B_r\setminus B_\delta)\big),\quad\text{for all
}0<\delta<r<1.
\end{equation}
We denote by~$\nu$ both the exterior normal at~$\partial\Omega$ and
the exterior normal at~$\partial B_r$, since no confusion can arise.
Testing the equation in~\eqref{MAIN EQ} against the solution itself, we see that
\begin{equation*}
\begin{split}
\int_{B_r\cap\Omega} gu\,&=
\int_{B_r\cap\Omega} {\rm div}(A\nabla u)\,u\\
&=\int_{B_r\cap\Omega} \Big( {\rm div}(Au\nabla u)-A|\nabla u|^2\Big)\\
&=\int_{\partial B_r\cap\Omega} Au\nabla u\cdot\nu+
\int_{B_r\cap\partial\Omega} Au\nabla u\cdot\nu
-\int_{B_r\cap\Omega} A|\nabla u|^2\\
&=\int_{\partial B_r\cap\Omega} Au\nabla u\cdot\nu+
\int_{B_r\cap\partial\Omega} fu
-\int_{B_r\cap\Omega} A|\nabla u|^2.
\end{split}
\end{equation*}
Hence, recalling~\eqref{DeH DEF},
\begin{equation}\label{INHPRIMO}
\int_{\partial B_r\cap\Omega} Au\nabla u\cdot\nu=r^{n-2} D(r).
\end{equation}
Using again~\eqref{MAIN EQ},
we also observe that
\begin{equation}\label{DIV CALC}
\begin{split}
&{\rm div}\left( A(\nabla u\cdot x)\nabla u-\frac{A}{2}|\nabla u|^2 x\right)-(\nabla u\cdot x)g
\\
= \;&(\nabla u\cdot x){\rm div}( A\nabla u)-(\nabla u\cdot x)g
+ A\nabla u\cdot\nabla(\nabla u\cdot x)
-\frac12{\rm div}(A|\nabla u|^2 x)
\\
= \;&
\sum_{i,j=1}^n \Big(A\partial_i u\,\partial_i(\partial_j u\,x_j)
-\frac12 \partial_i(A(\partial_ju)^2 x_i)\Big)\\
= \;&
\sum_{i,j=1}^n \Big( A\partial_i u\,\partial_{ij}^2 u \,x_j
+A(\partial_i u)^2\delta_{ij}
-\frac12 \partial_i A(\partial_ju)^2 x_i
-A\partial_ju\,\partial_{ij}^2u \,x_i
-\frac12 A(\partial_ju)^2 
\Big)\\
= \;&
\sum_{i,j=1}^n \Big( 
A(\partial_i u)^2\delta_{ij}
-\frac12 \partial_i A(\partial_ju)^2 x_i
-\frac12 A(\partial_ju)^2 
\Big)\\ =\;&
\frac{(2-n)\,A}{2}|\nabla u|^2-\frac12|\nabla u|^2\nabla A\cdot x.
\end{split}
\end{equation}
On the other hand, from~\eqref{DeH DEF} we know that
\begin{equation}\label{DPRIMO}
\begin{split}
D'(r)\,&= (2-n)r^{1-n}\int_{B_r\cap\Omega} A|\nabla u|^2
+r^{2-n}\int_{\partial B_r\cap\Omega} A|\nabla u|^2\\&\qquad
-(2-n)r^{1-n}\int_{B_r\cap\partial\Omega} fu
-r^{2-n}\int_{\partial B_r\cap\partial\Omega} fu\\&\qquad
+(2-n)r^{1-n}\int_{B_r\cap\Omega}gu
+r^{2-n}\int_{\partial B_r\cap\Omega}gu,
\end{split}
\end{equation}
and (recalling that $\Omega$ is a cone, hence~$\Omega/r=\Omega$ for each~$r>0$)
\begin{equation}\label{HPRIMO}
\begin{split}
H'(r)\,&= 
\int_{\Sigma} \nabla A(ry)\cdot y\,u^2(ry)
\,d{\mathcal{H}}^{n-1}_y
+
2\int_{\Sigma} A(ry)\,u(ry)\nabla u(ry)\cdot y
\,d{\mathcal{H}}^{n-1}_y
\\&= r^{1-n}
\int_{\partial B_r\cap\Omega} \nabla A\cdot \nu\,u^2
+
2r^{1-n}\int_{\partial B_r\cap\Omega} Au\nabla u\cdot\nu.
\end{split}
\end{equation}
Thus, comparing~\eqref{INHPRIMO} with~\eqref{HPRIMO} we conclude that
\begin{equation*}
H'(r)-r^{1-n}\int_{\partial B_r\cap\Omega} \nabla A\cdot \nu\,u^2=
2r^{1-n}\int_{\partial B_r\cap\Omega} Au\nabla u\cdot\nu
= \frac{2D(r)}{r},
\end{equation*}
and therefore
\begin{equation}\label{DNUOVA}
D(r)=\frac{rH'(r)}2-
\frac{r^{2-n}}2\int_{\partial B_r\cap\Omega} \nabla A\cdot \nu\,u^2.
\end{equation}
From \eqref{eq:regH2} it follows that, for all $0<\delta<r<1$, 
$A(\nabla u\cdot x)\nabla u-\frac{A}{2}|\nabla u|^2 x\in
W^{1,1}(\Omega\cap(B_r\setminus B_\delta)$ so that 
\begin{multline}\label{eq:4}
\int_{\Omega\cap(B_r\setminus B_\delta)}
{\rm div}\left( A(\nabla u\cdot x)\nabla u-\frac{A}{2}|\nabla u|^2
  x\right)=
\int_{\partial B_r\cap\Omega}r\,\bigg(
A(\nabla u\cdot \nu)^2
-\frac A2|\nabla u|^2 \bigg)\\
-
\int_{\partial B_\delta\cap\Omega}\delta\,\bigg(
A(\nabla u\cdot \nu)^2
-\frac A2|\nabla u|^2 \bigg)
+\int_{(B_r\setminus B_\delta)\cap\partial\Omega}
\Big( f\,\nabla u\cdot x-\tfrac A2|\nabla u|^2 x\cdot\nu\Big).
\end{multline}
Since 
\[
  \int_0^{1} \left[\int_{\Omega\cap\partial B_r}|\nabla
    u|^2\right]\,dr=
\int_{\Omega\cap B_1}|\nabla
    u|^2<+\infty,
\]
there exists a decreasing sequence $\{\delta_n\}\subset
(0,1)$ such that $\lim_{n\to +\infty} \delta_n=0$ and
\begin{equation*}
 \delta_n \int_{\Omega\cap\partial B_{\delta_n}}|\nabla
    u|^2\longrightarrow 0 \text{\quad as }
  n\to +\infty .
\end{equation*}
Choosing $\delta=\delta_n$ in \eqref{eq:4} and letting $n\to\infty$ we
then obtain 
\begin{multline*}
\int_{\Omega\cap B_r}
{\rm div}\left( A(\nabla u\cdot x)\nabla u-\frac{A}{2}|\nabla u|^2
  x\right)\\
=
\int_{\partial B_r\cap\Omega}r\,\bigg(
A(\nabla u\cdot \nu)^2
-\frac A2|\nabla u|^2 \bigg)
+\int_{B_r \cap\partial\Omega}
\Big( f\,\nabla u\cdot x-\tfrac A2|\nabla u|^2 x\cdot\nu\Big).
\end{multline*}
Therefore, taking into account \eqref{DIV CALC},
\begin{eqnarray*}
&& (2-n)r^{1-n}\int_{B_r\cap\Omega} A|\nabla u|^2\\
&=& 2r^{1-n}\int_{B_r\cap\Omega} \left[\frac12
|\nabla u|^2\nabla A\cdot x+
{\rm div}\left( A(\nabla u\cdot x)\nabla u-\frac{A}{2}|\nabla u|^2 x\right)-
(\nabla u\cdot x)g
\right]\\
&=& r^{1-n}\int_{B_r\cap\Omega} \Big(
|\nabla u|^2\nabla A\cdot x-2
(\nabla u\cdot x)g
\Big)
+
\int_{\partial B_r\cap\Omega}\Big(
2r^{-n}A(\nabla u\cdot x)^2
-r^{2-n}A|\nabla u|^2 \Big)\\&&\qquad
+r^{1-n}\int_{B_r\cap\partial\Omega}
\Big( 2 f\,\nabla u\cdot x-A|\nabla u|^2 x\cdot\nu\Big).
\end{eqnarray*}
We thereby substitute this identity into~\eqref{DPRIMO} and we conclude that
\begin{equation*}
\begin{split}
D'(r)\,&= r^{1-n}\int_{B_r\cap\Omega} \Big(
|\nabla u|^2\nabla A\cdot x-2
(\nabla u\cdot x)g
\Big)
+
2r^{-n}\int_{\partial B_r\cap\Omega}
A(\nabla u\cdot x)^2
\\&\qquad
+r^{1-n}\int_{B_r\cap\partial\Omega}
\Big( 2f\,\nabla u\cdot x-A|\nabla u|^2 x\cdot\nu\Big)\\
&\qquad
-(2-n)r^{1-n}\int_{B_r\cap\partial\Omega} fu
-r^{2-n}\int_{\partial B_r\cap\partial\Omega} fu
\\&\qquad
+(2-n)r^{1-n}\int_{B_r\cap\Omega}gu
+r^{2-n}\int_{\partial B_r\cap\Omega}gu.
\end{split}
\end{equation*}
{F}rom this and~\eqref{DNUOVA}, we find that
\begin{equation}\label{HHPD}
\begin{split}&
D'(r)H(r)-H'(r)D(r)\\
=\;& 
H(r)\left[
r^{1-n}\int_{B_r\cap\Omega} \Big(
|\nabla u|^2\nabla A\cdot x-2
(\nabla u\cdot x)g
\Big)
+
2r^{-n}\int_{\partial B_r\cap\Omega}
A(\nabla u\cdot x)^2\right.\\&\quad
+r^{1-n}\int_{B_r\cap\partial\Omega}
\Big( 2f\,\nabla u\cdot x-A|\nabla u|^2 x\cdot\nu\Big)\\&\quad
-(2-n)r^{1-n}\int_{B_r\cap\partial\Omega} fu
-r^{2-n}\int_{\partial B_r\cap\partial\Omega} fu
\\&\quad\left.
+(2-n)r^{1-n}\int_{B_r\cap\Omega}gu
+r^{2-n}\int_{\partial B_r\cap\Omega}gu
\right]\\
&\quad-\frac{r(H'(r))^2}2+
\frac{r^{2-n}\,H'(r)}2\int_{\partial B_r\cap\Omega} \nabla A\cdot \nu\,u^2.
\end{split}
\end{equation}
On the other hand, recalling~\eqref{HPRIMO}, we see that
\begin{equation*}
\begin{split}&
-\frac{r(H'(r))^2}2+
\frac{r^{2-n}\,H'(r)}2\int_{\partial B_r\cap\Omega} \nabla A\cdot \nu\,u^2\\
=\;& 
-\frac{r}2
\left(r^{1-n}
\int_{\partial B_r\cap\Omega} \nabla A\cdot \nu\,u^2
+
2r^{1-n}\int_{\partial B_r\cap\Omega} Au\nabla u\cdot\nu\right)^2\\
&\quad+
\frac{r^{2-n}}2\left(\int_{\partial B_r\cap\Omega} \nabla A\cdot \nu\,u^2\right)
\left(
r^{1-n}
\int_{\partial B_r\cap\Omega} \nabla A\cdot \nu\,u^2
+
2r^{1-n}\int_{\partial B_r\cap\Omega} Au\nabla u\cdot\nu\right)\\
=\;& 
-\frac{r^{3-2n}}2\left(\int_{\partial B_r\cap\Omega} \nabla A\cdot \nu\,u^2\right)^2
-2r^{3-2n}\left(\int_{\partial B_r\cap\Omega} \nabla A\cdot \nu\,u^2\right)
\left(\int_{\partial B_r\cap\Omega} Au\nabla u\cdot\nu\right)\\
&\quad-2r^{3-2n}\left(\int_{\partial B_r\cap\Omega} Au\nabla u\cdot\nu\right)^2 \\
&\quad+
\frac{r^{3-2n}}2\left(\int_{\partial B_r\cap\Omega} \nabla A\cdot \nu\,u^2\right)^2
+r^{3-2n}\left(\int_{\partial B_r\cap\Omega} \nabla A\cdot \nu\,u^2\right)
\left(\int_{\partial B_r\cap\Omega} Au\nabla u\cdot\nu\right)
\\=\;& 
-r^{3-2n}\left(\int_{\partial B_r\cap\Omega} \nabla A\cdot \nu\,u^2\right)
\left(\int_{\partial B_r\cap\Omega} Au\nabla u\cdot\nu\right)
-2r^{3-2n}\left(\int_{\partial B_r\cap\Omega} Au\nabla u\cdot\nu\right)^2 .
\end{split}
\end{equation*}
Hence, substituting this identity into~\eqref{HHPD}, we conclude that
\begin{equation}\label{PLUAQ}
\begin{split}&
D'(r)H(r)-H'(r)D(r)\\
=\;& 
H(r)\left[
r^{1-n}\int_{B_r\cap\Omega} \Big(
|\nabla u|^2\nabla A\cdot x-2
(\nabla u\cdot x)g
\Big)
+
2r^{-n}\int_{\partial B_r\cap\Omega}
A(\nabla u\cdot x)^2\right.\\&\quad
+r^{1-n}\int_{B_r\cap\partial\Omega}
\Big( 2f\,\nabla u\cdot x-A|\nabla u|^2 x\cdot\nu\Big)\\&\quad
-(2-n)r^{1-n}\int_{B_r\cap\partial\Omega} fu
-r^{2-n}\int_{\partial B_r\cap\partial\Omega} fu
\\&\quad\left.
+(2-n)r^{1-n}\int_{B_r\cap\Omega}gu
+r^{2-n}\int_{\partial B_r\cap\Omega}gu
\right]\\
&\quad
-r^{3-2n}\left(\int_{\partial B_r\cap\Omega} \nabla A\cdot \nu\,u^2\right)
\left(\int_{\partial B_r\cap\Omega} Au\nabla u\cdot\nu\right)
-2r^{3-2n}\left(\int_{\partial B_r\cap\Omega} Au\nabla u\cdot\nu\right)^2 .
\end{split}
\end{equation}
Moreover, from the Cauchy-Schwarz Inequality, we know that
$$ \int_{\partial B_r\cap\Omega} Au\nabla u\cdot x\le
\sqrt{\int_{\partial B_r\cap\Omega} Au^2
\int_{\partial B_r\cap\Omega} A(\nabla u\cdot x)^2}\,.$$
Consequently,
using again~\eqref{DeH DEF}, we also observe that
\begin{eqnarray*}
&& 2r^{-n} H(r)
\int_{\partial B_r\cap\Omega}
A(\nabla u\cdot x)^2-
2r^{3-2n}\left(\int_{\partial B_r\cap\Omega} Au\nabla u\cdot\nu\right)^2
\\ &=& 2r^{1-2n}\left[\left(\int_{\partial B_r\cap\Omega} A\,u^2\right)
\left(\int_{\partial B_r\cap\Omega} A(\nabla u\cdot x)^2\right)
-\left(\int_{\partial B_r\cap\Omega} Au\nabla u\cdot x\right)^2\right]\\&\ge&0.
\end{eqnarray*}
Plugging this information into~\eqref{PLUAQ}, we thus obtain that
\begin{equation}\label{0102303-A}
\begin{split}&
D'(r)H(r)-H'(r)D(r)\\
\ge\;& 
H(r)\left[
r^{1-n}\int_{B_r\cap\Omega} \Big(
|\nabla u|^2\nabla A\cdot x-2
(\nabla u\cdot x)g
\Big)
\right.\\&\quad
+r^{1-n}\int_{B_r\cap\partial\Omega}
\Big( 2f\,\nabla u\cdot x-A|\nabla u|^2 x\cdot\nu\Big)\\&\quad
-(2-n)r^{1-n}\int_{B_r\cap\partial\Omega} fu
-r^{2-n}\int_{\partial B_r\cap\partial\Omega} fu
\\&\quad\left.
+(2-n)r^{1-n}\int_{B_r\cap\Omega}gu
+r^{2-n}\int_{\partial B_r\cap\Omega}gu
\right]\\
&\quad
-r^{3-2n}\left(\int_{\partial B_r\cap\Omega} \nabla A\cdot \nu\,u^2\right)
\left(\int_{\partial B_r\cap\Omega} Au\nabla u\cdot\nu\right).
\end{split}
\end{equation}
Then, from~\eqref{0102303-A} and~\eqref{0102303}, we obtain that
\begin{equation}\label{FONDA}
\begin{split}&
D'(r)H(r)-H'(r)D(r)\\
\ge\;& 
H(r)\left[
r^{1-n}\int_{B_r\cap\Omega} \Big(
|\nabla u|^2\nabla A\cdot x-2
(\nabla u\cdot x)g
\Big)
\right.\\&\quad
+2r^{1-n}\int_{B_r\cap\partial\Omega}f\,\nabla u\cdot x
-(2-n)r^{1-n}\int_{B_r\cap\partial\Omega} fu
-r^{2-n}\int_{\partial B_r\cap\partial\Omega} fu
\\&\quad\left.
+(2-n)r^{1-n}\int_{B_r\cap\Omega}gu
+r^{2-n}\int_{\partial B_r\cap\Omega}gu\right]\\
&\quad
-r^{3-2n}\left(\int_{\partial B_r\cap\Omega} \nabla A\cdot \nu\,u^2\right)
\left(\int_{\partial B_r\cap\Omega} Au\nabla u\cdot\nu\right).\end{split}
\end{equation}
Now, we define
\begin{equation}\label{DEE} E(r):=r^{2-n}
\int_{B_r\cap\Omega} A|\nabla u|^2.\end{equation}
By~\eqref{STR:HY2},
we have that
\begin{equation}\label{9od7-A}
\left| \int_{B_r\cap\partial\Omega} fu\right|\le C
\int_{B_r\cap\partial\Omega} A\,|x|^{\delta-1}\,|u|^2.\end{equation}
On the other hand, by
Lemma~\ref{TRACEIN1} (used here with~$\gamma:=1-\delta$),
we see that
$$ \int_{B_r\cap\partial\Omega} A\,|x|^{\delta-1}\,|u|^2\le C
\,\int_{\Omega\cap B_r}\left[
r^\delta \,A\,|\nabla u|^2
+\frac{A\,u^2}{|x|^{2-\delta}}\right].$$
Hence, in view of Corollary \ref{cor:dis} (used here with~$\mu:=2-\delta$), \eqref{DeH DEF}
and~\eqref{DEE}
\begin{equation}\label{tobepFAL}\begin{split}
\int_{B_r\cap\partial\Omega} A\,|x|^{\delta-1}\,|u|^2\,&\le
C
\,r^\delta\int_{\Omega\cap B_r}
A\,|\nabla u|^2
+\frac{C}{r^{1-\delta}}\int_{\partial B_r\cap\Omega}Au^2
\\ &\le
C\, r^{n-2+\delta} \big(H(r)+E(r)\big).
\end{split}\end{equation}
Therefore, in light of \eqref{9od7-A}
\begin{equation}\label{tobep}
\left| \int_{B_r\cap\partial\Omega} fu\right|
\le C\, r^{n-2+\delta} \big(H(r)+E(r)\big).\end{equation}
Also, by \eqref{g:grow} and
Corollary \ref{cor:dis} (used here with~$\mu:=2-\delta$),
\begin{equation}\label{tobepg}\begin{split} &
\int_{B_r\cap \Omega} |g|\,|u|\le
C \int_{B_r\cap \Omega} A|x|^{\delta-2} \,|u|^2
\\&\qquad\qquad\le
C
\,r^\delta\int_{\Omega\cap B_r}
A\,|\nabla u|^2
+\frac{C}{r^{1-\delta}}\int_{\partial B_r\cap\Omega}Au^2
\le
C\, r^{n-2+\delta} \big(H(r)+E(r)\big)
.\end{split}\end{equation}
Consequently, by~\eqref{tobep} and~\eqref{tobepg}
\begin{equation*}
E(r)-D(r)\le
|D(r)-E(r)|\le
r^{2-n}\left|\int_{B_r\cap\partial\Omega} fu-\int_{B_r\cap \Omega} gu
\right|\le
C\, r^\delta \big(H(r)+E(r)\big),
\end{equation*}
and therefore, for any~$r\in(0,1)$ sufficiently small,
\begin{equation}\label{E vs D}
\frac{E(r)}2\le (1-Cr^\delta)\,E(r)\le Cr^\delta H(r)+D(r).
\end{equation}
Estimate \eqref{E vs D} implies statement (i) with $r_0>0$
so small as to satisfy condition \eqref{E vs D} and $Cr_0^\delta<1$. Indeed, let us argue by
contradiction and assume that there exists $\bar r\in
(0,r_0)$ such that $H(\bar r)=0$. By \eqref{DeH DEF} this would
imply that $u\equiv 0$ on $\Omega\cap\partial B_{\bar r}$ and hence,
in view of \eqref{INHPRIMO}, $D(\bar r)=0$. Then \eqref{E vs D} yields
that $E(\bar r)=0$ and hence $u$ is constant in $\Omega\cap  B_{\bar
  r}$. Therefore $u\equiv 0$ in $\Omega\cap  B_{\bar
  r}$, which is in contradiction with \eqref{NONTRIVIAL}.
  
Furthermore, for all $r\in(0,r_0)$, \eqref{E vs D} implies that 
\begin{equation}\label{Epci}
0\leq
\frac{E(r)}2<H(r)+D(r),\end{equation} and hence $\mathcal N(r)+1>0$.

Moreover, from the
Sobolev Trace Theorem on manifolds applied on the spherical cap
$\partial B_r\cap\partial\Omega=r\partial\Sigma$, 
we have that, recalling \eqref{MUCK},
\begin{equation}\label{firebj}
\begin{split} 
\int_{\partial B_r\cap\partial\Omega} A\,|u|^2&\le
\frac1c\int_{\partial B_r\cap\partial\Omega} |u|^2=\frac{r^{n-2}}c\int_{\partial \Sigma} |u(r\theta)|^2\\
&\leq 
Cr^{n-2}\int_{\Sigma}\left(u^2(r\theta)+|\nabla(u^2(r\theta)|\right)\\
&\leq 
Cr^{n-2}\int_{\Sigma}\left(u^2(r\theta)+2r|u(r\theta)||\nabla u(r\theta)|\right)\\
&\leq
Cr^{n-2}H(r)+2Cr^{n-1}\sqrt{\int_{\Sigma}u^2(r\theta)}\sqrt{\int_{\Sigma}|\nabla
  u(r\theta)|^2}
\\
&\leq 
Cr^{n-2}H(r)+Cr^{n-\frac32}\sqrt{H(r)}\sqrt{r^{2-n}\int_{\Omega\cap \partial B_r}A |\nabla u|^2}
\\&\le
C\left(r\int_{\partial B_r\cap\Omega} 
A\,|\nabla u|^2+r^{n-2} H(r)\right)
\end{split}\end{equation}
for some $C>0$ independent of $r$ (varying from line to line).
Now, we recall~\eqref{STR:HY2} 
and we observe that
\begin{equation}\label{AQW-1}
\left|
\int_{\partial B_r\cap\partial\Omega}fu\right|\le C
\int_{\partial B_r\cap\partial\Omega}A\,|x|^{\delta-1}\,|u|^2=
Cr^{\delta-1}\int_{\partial B_r\cap\partial\Omega}A\,|u|^2
.\end{equation}
In addition, from~\eqref{g:grow},
\begin{equation}\label{AQW-2}
\left| \int_{\partial B_r\cap \Omega}gu\right|\le
C \int_{\partial B_r\cap \Omega}A\,|x|^{\delta-2}|u|^2=
C r^{\delta-2}\int_{\partial B_r\cap \Omega}A\,|u|^2=C r^{n+\delta-3}H(r).
\end{equation}
{F}rom~\eqref{DPRIMO}, \eqref{AQW-1} and~\eqref{AQW-2}, we obtain that 
\begin{equation*}
\begin{split}
  r^{2-n}\int_{\partial B_r\cap\Omega}A|\nabla
  u|^2&=D'(r)+\frac{n-2}{r}D(r)+r^{2-n}\int_{\partial
  B_r\cap\partial\Omega}fu
  -r^{2-n}\int_{\partial B_r\cap \Omega}gu
  \\
&\leq 
D'(r)+\frac{n-2}{r}D(r)+C\left(
r^{\delta+1-n}\int_{\partial B_r\cap\partial\Omega}A\,| u|^2+r^{\delta-1} H(r)\right)
.\end{split}\end{equation*}
Then from \eqref{firebj} it follows that 
\begin{equation*}
  r^{2-n}\int_{\partial B_r\cap\Omega}A|\nabla
  u|^2\leq 
D'(r)+\frac{n-2}{r}D(r)+C r^{\delta-1} H(r)+
C
r^{\delta+2-n}\int_{\partial B_r\cap\Omega}A\,|\nabla u|^2,
\end{equation*}
from which it follows that 
\begin{equation}\label{lsnUA}
  r^{2-n}\int_{\partial B_r\cap\Omega}A|\nabla
  u|^2\leq 
C\left(D'(r)+\frac{n-2}{r}D(r)+C r^{\delta-1} H(r)\right),
\end{equation}
for some $C>0$ and for all $r>0$ sufficiently small.

Plugging~\eqref{lsnUA} into~\eqref{firebj} we conclude that
\begin{equation}\label{lsnUA2}
\begin{split} 
\int_{\partial B_r\cap\partial\Omega} A\,|u|^2&\le 
Cr^{n-2}H(r)+Cr^{n-\frac32}\sqrt{H(r)}\sqrt{r^{2-n}\int_{\Omega\cap \partial
    B_r}A |\nabla u|^2}\\
&\le Cr^{n-2}H(r)+Cr^{n-\frac32}\sqrt{H(r)}\sqrt{
D'(r)+\tfrac{n-2}{r}D(r)+C r^{\delta-1} H(r)}
\end{split}
\end{equation}
as long as~$r$ is sufficiently small.
It is now our goal to use the previously obtained
information in order to estimate the right hand side of~\eqref{FONDA}.
To this end, we first observe that, from \eqref{STR:HY1},
\begin{eqnarray*}&&
r^{3-2n}\,\left| \left(\int_{\partial B_r\cap\Omega} \nabla A\cdot \nu\,u^2\right)
\left(\int_{\partial B_r\cap\Omega} Au\nabla u\cdot\nu\right)\right|\\&
\le& \e_r\,r^{2-2n}\,\left|\left( \int_{\partial B_r\cap\Omega} A\,u^2\right)
\left(\int_{\partial B_r\cap\Omega} Au\nabla u\cdot\nu\right)\right|\\
&=&
\e_r\,r^{1-n}\,H(r)\,\left|
\int_{\partial B_r\cap\Omega} Au\nabla u\cdot\nu\right|\\&\le&
\e_r\,r^{1-n}\,H(r)\,\sqrt{
\int_{\partial B_r\cap\Omega} A\,|u|^2}
\,\sqrt{
\int_{\partial B_r\cap\Omega} A\,|\nabla u|^2}\\&=&
\e_r\,r^{\frac{1-n}2}\,\big(H(r)\big)^{\frac32}
\,\sqrt{
\int_{\partial B_r\cap\Omega} A\,|\nabla u|^2}.
\end{eqnarray*}
This and~\eqref{lsnUA} lead to
\begin{equation}\label{UNT1}
\begin{split}
r^{3-2n}\,&\left|\left( \int_{\partial B_r\cap\Omega} \nabla A\cdot \nu\,u^2\right)
\left(\int_{\partial B_r\cap\Omega} Au\nabla u\cdot\nu\right)\right|
\\ &\le C
\e_r\,r^{-\frac{1}2}\,\big(H(r)\big)^{\frac32}
\,\sqrt{D'(r)+\tfrac{n-2}{r}D(r)+C r^{\delta-1} H(r)}.
\end{split}\end{equation}
Furthermore, by~\eqref{STR:HY1} and~\eqref{g:grow},
\begin{eqnarray*}&&
r^{1-n}\,\left|
\int_{B_r\cap\Omega} \Big(
|\nabla u|^2\nabla A\cdot x-2
(\nabla u\cdot x)g
\Big)\right|\\&&\qquad\le Cr^{1-n}\,
\int_{B_r\cap\Omega} \Big(
\e_r \,A|\nabla u|^2+A\,|x|^{\delta-1}|\nabla u|\,|u|
\Big)\\&&\qquad
\le Cr^{1-n}\,\left(
\e_r\int_{B_r\cap\Omega}A\,|\nabla u|^2+
\sqrt{\int_{B_r\cap\Omega}
A\,|\nabla u|^2}\,\sqrt{\int_{B_r\cap\Omega}
A\,|x|^{2(\delta-1)}|u|^2}\right)\\
&&\qquad=
Cr^{1-n}\,\left(
\e_r r^{n-2}E(r)+r^{\frac{n-2}{2}}
\sqrt{E(r)}\,\sqrt{\int_{B_r\cap\Omega}
A\,|x|^{2(\delta-1)}|u|^2}\right).
\end{eqnarray*}
Consequently, exploiting Corollary \ref{cor:dis} with~$\mu:=2(1-\delta)$,
\begin{equation*}
\begin{split}&
r^{1-n}\,\left|
\int_{B_r\cap\Omega} \Big(
|\nabla u|^2\nabla A\cdot x-2
(\nabla u\cdot x)g
\Big)\right|\\
\le\;&
Cr^{1-n}\,\left(
\e_r r^{n-2}E(r)+r^{\frac{n-2}{2}}
\sqrt{E(r)}\,\sqrt{
r^{2\delta-1}\int_{\partial B_r\cap\Omega}
A\,|u|^2+r^{2\delta}\int_{B_r\cap\Omega} A\,|\nabla u|^2}\right)
\\ \le\;&
C\,\left(\frac{\e_r}{r}\,E(r)+r^{\delta-1}
\sqrt{E(r)}\,\sqrt{H(r)+E(r)}\right).
\end{split}\end{equation*}
Now, plugging the latter inequality,
\eqref{tobep}, \eqref{tobepg}, \eqref{AQW-2} and~\eqref{UNT1}
into~\eqref{FONDA}, we conclude that
\begin{equation}\label{FOND567ABIS}
\begin{split}&
D'(r)H(r)-H'(r)D(r)\\
\ge\;&
H(r)\Bigg[
-C\,\left(\frac{\e_r}{r}\,E(r)+r^{\delta-1}
\sqrt{E(r)}\,\sqrt{H(r)+E(r)}\right)
\\&\quad
+2r^{1-n}\int_{B_r\cap\partial\Omega}f\,\nabla u\cdot x
-C\, r^{\delta-1} \big(H(r)+E(r)\big)
-r^{2-n}\int_{\partial B_r\cap\partial\Omega} fu
\\&\quad
-C\, r^{\delta-1} \big(H(r)+E(r)\big)
-C r^{\delta-1}H(r)\Bigg]\\
&\quad
-C
\e_r\,r^{-\frac{1}2}\,\big(H(r)\big)^{\frac32}
\,\sqrt{D'(r)+\tfrac{n-2}{r}D(r)+C r^{\delta-1} H(r)}\\
\ge\;&
H(r)\Bigg[
-C\,\left(\frac{\e_r}{r}\,E(r)+r^{\delta-1}
\sqrt{E(r)}\,\sqrt{H(r)+E(r)}\right)-C\, r^{\delta-1} \big(H(r)+E(r)\big)
\\&\quad
+2r^{1-n}\int_{B_r\cap\partial\Omega}f\,\nabla u\cdot x
-r^{2-n}\int_{\partial B_r\cap\partial\Omega} fu\Bigg]\\
&\quad
-C
\e_r\,r^{-\frac{1}2}\,\big(H(r)\big)^{\frac32}
\,\sqrt{D'(r)+\tfrac{n-2}{r}D(r)+C r^{\delta-1} H(r)},
\end{split}
\end{equation}
for some~$C>0$.

Now, recalling~\eqref{AQW-1} and~\eqref{lsnUA2}, we notice that
\begin{equation*}
\left|
\int_{\partial B_r\cap\partial\Omega}fu\right|\le 
C\,\Big(
r^{n+\delta-3} H(r)+
r^{n+\delta-\frac52}\sqrt{H(r)}\sqrt{
D'(r)+\tfrac{n-2}{r}D(r)+C r^{\delta-1} H(r)}
\Big).
\end{equation*}
This and~\eqref{FOND567ABIS} give that
\begin{equation}\label{FONDABIS}
\begin{split}&
D'(r)H(r)-H'(r)D(r)\\
\ge\;&
H(r)\Bigg[
-C\,\left(\frac{\e_r}{r}\,E(r)+r^{\delta-1}
\sqrt{E(r)}\,\sqrt{H(r)+E(r)}\right)-C\, r^{\delta-1} \big(H(r)+E(r)\big)
\\&\quad
+2r^{1-n}\int_{B_r\cap\partial\Omega}f\,\nabla u\cdot x
-C\, r^{\delta-1} H(r)-C
r^{\delta-\frac12}\sqrt{H(r)}\sqrt{
D'(r)+\tfrac{n-2}{r}D(r)+C r^{\delta-1} H(r)}\Bigg]\\
&\quad
-C
\e_r\,r^{-\frac{1}2}\,\big(H(r)\big)^{\frac32}
\,\sqrt{D'(r)+\tfrac{n-2}{r}D(r)+C r^{\delta-1} H(r)}\\
\ge\;& 
H(r)\Bigg[
-C\,\left(\frac{\e_r}{r}\,E(r)+r^{\delta-1}
\sqrt{E(r)}\,\sqrt{H(r)+E(r)}\right)-C\, r^{\delta-1} \big(H(r)+E(r)\big)\\
&\qquad+2r^{1-n}\int_{B_r\cap\partial\Omega}f\,\nabla u\cdot x
\Bigg]\\
&\quad
-C\max\{r^\delta,
\e_r\}\,r^{-\frac{1}2}\,\big(H(r)\big)^{\frac32}
\,\sqrt{D'(r)+\tfrac{n-2}{r}D(r)+C r^{\delta-1} H(r)}.
\end{split}
\end{equation}
Now, we denote by~$\partial_\star:=\frac{x}{|x|}\cdot\nabla$
and we observe that~$\partial_\star$ is the ``radial'' component of the tangential
gradient along~$\partial\Omega$, since~$\Omega$ is a cone.
Hence, since, by~\eqref{HYhstar},
$$ \nabla \big( F(x,u(x))\big)=\nabla\left( \int_0^{u(x)} f(x,\tau)\,d\tau\right)=
\int_0^{u(x)} \nabla_x f(x,\tau)\,d\tau+ f(x,u(x))\nabla u(x),$$
we obtain that
$$ |x|\,\partial_\star \big( F(x,u(x))\big)=
\int_0^{u(x)} \nabla_x f(x,\tau)\cdot x\,d\tau+ f(x,u(x))\nabla u(x)\cdot x.$$
As a consequence, by~\eqref{STR:HY2DER},
\begin{equation}\label{78JS34756}
\begin{split}
f(x,u(x))\nabla u(x)\cdot x \;&\le\;
|x|\,\partial_\star \big( F(x,u(x))\big)+
C\int_0^{|u(x)|} A(x)\,|x|^{\delta-1}\,\tau\,d\tau
\\&\le\;
|x|\,\partial_\star \big( F(x,u(x))\big)+ 
C\,A(x)\,|x|^{\delta-1}\,|u(x)|^2.
\end{split}\end{equation}
Moreover, integrating by parts along~$\partial\Omega$,
\begin{equation}\label{76744823}\left|
\int_{B_r\cap\partial\Omega} |x|\,\partial_\star \big( F(x,u(x))\big)\right|
\le C\int_{B_r\cap\partial\Omega} \big| F(x,u(x))\big|+
C\int_{\partial(B_r\cap\partial\Omega)} |x|\,\big| F(x,u(x)\big|.
\end{equation}
In addition, by~\eqref{STR:HY2} and~\eqref{HYhstar},
we know that
$$ |F(x,t)|\le C\,A(x)\,|x|^{\delta-1}\,\int_0^{|t|} \tau\,d\tau\le
C\,A(x)\,|x|^{\delta-1}\,|t|^2.
$$
This and~\eqref{76744823} lead to
$$
\left|
\int_{B_r\cap\partial\Omega} |x|\,\partial_\star \big( F(x,u(x))\big)\right|
\le C\int_{B_r\cap\partial\Omega} A(x)\,|x|^{\delta-1}\,|u(x)|^2+
C\int_{\partial B_r\cap\partial\Omega}
A(x)\,|x|^{\delta}\,|u(x)|^2
.$$
Hence, recalling~\eqref{78JS34756},
\begin{equation*}
\begin{split}\left|
\int_{B_r\cap\partial\Omega} f(x,u(x))\nabla u(x)\cdot x \right|\;&\le\;\left|
\int_{B_r\cap\partial\Omega}
|x|\,\partial_\star \big( F(x,u(x))\big)\right|+ C\int_{B_r\cap\partial\Omega}
A(x)\,|x|^{\delta-1}\,|u(x)|^2\\ &\le\;
C\int_{B_r\cap\partial\Omega} A(x)\,|x|^{\delta-1}\,|u(x)|^2+
C\int_{\partial B_r\cap\partial\Omega}
A(x)\,|x|^{\delta}\,|u(x)|^2,
\end{split}
\end{equation*}
up to renaming~$C>0$.

Therefore, recalling~\eqref{tobepFAL} and~\eqref{lsnUA2},
\begin{multline*}
\left|
\int_{B_r\cap\partial\Omega} f(x,u(x))\nabla u(x)\cdot x \right|\\\le
C\, r^{n-2+\delta} \Big(H(r)+E(r)\Big)+
Cr^{n+\delta-\frac32}\sqrt{H(r)}\sqrt{
D'(r)+\tfrac{n-2}{r}D(r)+C r^{\delta-1} H(r)}
.
\end{multline*}
Then, we insert this information into~\eqref{FONDABIS}
and we conclude that
\begin{equation*}
\begin{split}
D'(r)&H(r)-H'(r)D(r)\\
\ge\;&
H(r)\Bigg[
-C\,\left(\frac{\e_r}{r}\,E(r)+r^{\delta-1}
\sqrt{E(r)}\,\sqrt{H(r)+E(r)}\right)-C\, r^{\delta-1} \big(H(r)+E(r)\big)\\
&\qquad
-C\, r^{\delta-1} \big(H(r)+E(r)\big)
-Cr^{\delta-\frac12}\sqrt{H(r)}\sqrt{
D'(r)+\tfrac{n-2}{r}D(r)+C r^{\delta-1} H(r)}
\Bigg]\\
&\quad
-C\max\{r^\delta,
\e_r\}\,r^{-\frac{1}2}\,\big(H(r)\big)^{\frac32}
\,\sqrt{D'(r)+\tfrac{n-2}{r}D(r)+C r^{\delta-1} H(r)}\\
\ge\;&
H(r)\Bigg[
-C\,\left(\frac{\e_r}{r}\,E(r)+r^{\delta-1}
\sqrt{E(r)}\,\sqrt{H(r)+E(r)}\right)-C\, r^{\delta-1} \big(H(r)+E(r)\big) \Bigg]\\
&\quad
-C\max\{r^\delta,
\e_r\}\,r^{-\frac{1}2}\,\big(H(r)\big)^{\frac32}
\,\sqrt{D'(r)+\tfrac{n-2}{r}D(r)+C r^{\delta-1} H(r)}.
\end{split}
\end{equation*}
Accordingly, by~\eqref{DeH DEN},
\begin{equation*}\begin{split}
{\mathcal{N}}'(r)\,&
=
\frac{d}{dr}\left(\frac{D(r)}{H(r)}\right)=
\frac{ D'(r)H(r)-H'(r)D(r) }{H^2(r)}\\&
\ge
-C\,\left(\frac{\e_r}{r}\,\frac{E(r)}{H(r)}+r^{\delta-1}
\sqrt{\frac{E(r)}{H(r)}}\,\sqrt{1+\frac{E(r)}{H(r)}}\right)-C\, r^{\delta-1} \left(
1+\frac{E(r)}{H(r)}\right)
\\&\quad
-C\max\{r^\delta,
\e_r\}\,r^{-\frac{1}2}
\,\sqrt{\frac{D'(r)}{H(r)}+\frac{n-2}{r}\frac{D(r)}{H(r)}+C r^{\delta-1}}.
\end{split}
\end{equation*}
{F}rom this inequality and~\eqref{E vs D} we find that
\begin{equation}\label{eq:2}
  \begin{split}
    {\mathcal{N}}'(r)
    &\ge
    -C\,\left(\frac{\e_r}{r}\,(1+\mathcal N(r))+r^{\delta-1}
      \sqrt{1+\mathcal N(r)}\,\sqrt{2+\mathcal N(r)}\right)-C\, r^{\delta-1} \left(
      2+\mathcal N(r)\right)
    \\&\quad
    -C\max\{r^\delta,
    \e_r\}\,r^{-\frac{1}2}
    \,\sqrt{\frac{D'(r)}{H(r)}+\frac{n-2}{r}\frac{D(r)}{H(r)}+C
      r^{\delta-1}}\\
    &\ge
    -C\,\frac{\e_r}r (1+\mathcal N(r)) -C\,r^{\delta-1}(1+\mathcal N(r)) -C\,r^{\delta-1}\\
    & \quad-C\,\max\{r^\delta,\e_r\}r^{-\frac 12}\sqrt{\frac{D'(r)}{H(r)}+\frac{n-2}{r}\frac{D(r)}{H(r)}+C
      r^{\delta-1}}\\
    & \ge
    -C\,\max\{r^\delta,\e_r\}r^{-1}\left[(2+\mathcal N(r))+
      \sqrt{r\frac{D'(r)}{H(r)}+(n-2)\frac{D(r)}{H(r)}+C
        r^{\delta}}\right].
  \end{split}
\end{equation}
Let 
\begin{equation*}
\Lambda=\{r\in(0,r_0):D'(r)H(r)\le D(r)H'(r)\}.
\end{equation*}
In view of \eqref{DNUOVA}, \eqref{eq:H-N+1-pos}, and \eqref{STR:HY1}, for $r\in\Lambda$ we can estimate
$D'(r)$ as follows:
\begin{align*}
  D'(r)&\leq \frac{D(r)H'(r)}{H(r)}=\frac{2}r\frac{D^2(r)}{H(r)}+
r^{1-n}(\mathcal N(r)+1)\int_{\partial B_r\cap\Omega} \nabla A\cdot \nu\,u^2 -
r^{1-n}\int_{\partial B_r\cap\Omega} \nabla A\cdot \nu\,u^2\\
&\leq \frac{2}r\frac{D^2(r)}{H(r)}+
(\mathcal N(r)+1)\frac{\e_r}rH(r)+\frac{\e_r}rH(r)=
\frac{2}r\frac{D^2(r)}{H(r)}+
(\mathcal N(r)+2)\frac{\e_r}rH(r).
\end{align*}
It follows that, for all $r\in\Lambda$,
\begin{align*}
 \sqrt{r\frac{D'(r)}{H(r)}+(n-2)\frac{D(r)}{H(r)}+C
    r^{\delta}}&\leq
\sqrt{2\mathcal N^2(r)+\e_r (\mathcal N(r)+2)+(n-2)\mathcal N(r) +C
    r^{\delta}} \\
&\leq C (\mathcal N(r)+2).
\end{align*}
Combining the previous estimate with \eqref{eq:2} we obtain that, for
all $r\in\Lambda$ sufficiently small
\begin{equation}\label{eq:3}
    {\mathcal{N}}'(r)\geq -C\,\max\{r^\delta,\e_r\}r^{-1}(2+\mathcal N(r)).
\end{equation}
For $r\not\in \Lambda$ estimate \eqref{eq:3} is trivial, since the
left hand side of \eqref{eq:3} is nonnegative outside $\Lambda$ whereas the
right hand side is nonpositive because of \eqref{eq:H-N+1-pos}.
 Estimate \eqref{eq:stima-N'} and statement (ii) are thereby proved.

To prove statement (iii), let $h(r):=\max\{r^\delta,\e_r\}r^{-1}$. By
assumption~\eqref{L1}, we have that $h\in L^1(0,r_1)$. Then, from \eqref{eq:stima-N'}
it follows that 
\begin{equation*}
  \left((2+\mathcal N(r))e^{-C_1\int_r^1 h(s)\,ds}\right)'=e^{-C_1\int_r^1 h(s)\,ds}\Big({\mathcal{N}}'(r)+C_1h(r) (2+\mathcal N(r))\Big)\geq0
\end{equation*}
hence the function $w(r):= (2+\mathcal N(r))e^{-C_1\int_r^1
  h(s)\,ds}$ is nondecreasing in $(0,r_1)$.
  
Moreover $w\geq 0$ in view
of \eqref{eq:H-N+1-pos}. Therefore $w$  admits a finite limit as $r\to
0^+$ and then also $\mathcal N$ has a finite limit $\gamma$ as $r\to
0^+$.
Since estimate \eqref{E vs D} implies that $\mathcal N(r)\geq -C
r^\delta$ in $(0,r_0)$, we conclude that $\gamma\geq0$.

\section{Proof of Theorem~\ref{UQ:1}}\label{9SIoryuitis}

We start by proving~\eqref{UQ:3}. To this end, we argue for a contradiction
and we suppose that~\eqref{UQ:3} is violated. Then, we have that~\eqref{NONTRIVIAL}
is satisfied and hence all the hypotheses of Theorem~\ref{MONOTONIA} are
fulfilled. In particular, by the fact that the limit in
\eqref{LIMITgamma} is finite and $\mathcal N$ is continuous in $(0,r_0)$, we find that $\mathcal N$ is bounded,
i.e.  
for all~$r\in(0,r_0)$,
\begin{equation}\label{LAois}
{\mathcal{N}}(r)\le C,
\end{equation}
for some $C>0$.

Moreover, by~\eqref{DNUOVA},
$$ \frac{2D(r)}{r H(r)}=\frac{H'(r)}{H(r)}-
\frac{r^{1-n}}{H(r)}\int_{\partial B_r\cap\Omega} \nabla A\cdot \nu\,u^2.
$$
As a consequence, recalling~\eqref{TANGE},
\begin{eqnarray*}
\frac{H'(r)}{H(r)}&=&\frac{2D(r)}{r H(r)}+
\frac{r^{1-n}}{H(r)}\int_{\partial B_r\cap\Omega} \nabla A\cdot \nu\,u^2\\
&=&
\frac{2{\mathcal{N}}(r)}{r}+
\frac{r^{1-n}}{H(r)}\int_{\partial B_r\cap\Omega} \nabla A\cdot \nu\,u^2\\
&\le&
\frac{2{\mathcal{N}}(r)}{r}+
C\,\frac{r^{1-n}}{r H(r)}\int_{\partial B_r\cap\Omega} A\,u^2\\
&=&
\frac{2{\mathcal{N}}(r)}{r}+
\frac{C}{r}\\&\le&
\frac Cr\big( {\mathcal{N}}(r)+1\big),
\end{eqnarray*}
for some $C>0$ independent of $r$ (varying from line to line).
This and~\eqref{LAois} yield that
\begin{equation}\label{eq:H'suH}
\frac{H'(r)}{H(r)}\le \frac{C}{r},
\end{equation}
up to renaming~$C>0$ and therefore, if~$r\in(0,r_0/2)$,
\begin{equation}\label{DOU}
\begin{split}
\frac{H(2r)}{H(r)}\,& = \exp\left( \log H(2r)-\log H(r)\right)\\
&=\exp \left( \int_r^{2r} \frac{H'(\rho)}{H(\rho)}\,d\rho\right)\\
&\le \exp \left( C\,\int_r^{2r} \frac{d\rho}\rho\right)\\&=C,
\end{split}
\end{equation}
up to renaming~$C$ line after line. More in general, integration of
\eqref{eq:H'suH} over the interval $(r,rR)$ yields that for every
$R>1$ there exists $C_R>0$ (depending on $R$ but independent of $r$) such that 
\begin{equation}\label{eq:doub_gen}
  H(Rr)\leq C_R H(r)\quad\text{for all }r\in(0,r_0/R).
\end{equation}

The inequality in~\eqref{DOU} provides a pivotal ``doubling property''
in our setting. {F}rom this, we obtain that
$$ \int_{\partial B_{2r}\cap\Omega} A(x)\,u^2(x)
\,d{\mathcal{H}}^{n-1}_x\le
C\,\int_{\partial B_r\cap\Omega} A(x)\,u^2(x)
\,d{\mathcal{H}}^{n-1}_x,$$
up to renaming~$C$.

Integrating the latter inequality in~$r$, we find that
$$ \int_{B_{2r}\cap\Omega} A(x)\,u^2(x)
\,dx\le
C_0\,\int_{ B_r\cap\Omega} A(x)\,u^2(x)
\,dx,$$
for some~$C_0>0$ independent of $r$, which gives that
\begin{equation}\label{Lam}
\int_{B_{r}\cap\Omega} A(x)\,u^2(x)
\,dx\le C_0^m \int_{B_{r/2^m}\cap\Omega} A(x)\,u^2(x)
\,dx,
\end{equation}
for all~$m\in\N$ and $r\in(0,r_0)$.

Now we fix~$k\in\N$ such that~$2^{2k}\ge 2C_0$. In light of~\eqref{UQ:2}
we can write that
$$ |u(x)|\le |x|^k,$$
as long as~$x\in\Omega$ and~$|x|$ is sufficiently small.
Hence, we can exploit~\eqref{Lam} for~$m$ sufficiently large and conclude that
\begin{eqnarray*}&&
\int_{B_{r_0}\cap\Omega} A(x)\,u^2(x)
\,dx\le C_0^m \int_{B_{r_0/2^m}\cap\Omega} A(x)\,|x|^{2k}
\,dx\\&&\qquad\le
C_0^m \,\left(\frac{r_0}{2^m}\right)^{2k}\int_{B_{r_0/2^m}\cap\Omega} A(x)\,dx
\le
\frac{r_0^{2k}}{2^m} \,\|A\|_{L^1(\Omega)}.
\end{eqnarray*}
Then, sending~$m\to+\infty$, we conclude that
$$ \int_{B_{r_0}\cap\Omega} A(x)\,u^2(x)
\,dx=0,$$
and therefore, by~\eqref{MUCK}, it follows that~$u$
must vanish necessarily in~$B_{r_0}\cap\Omega$.
This proves~\eqref{UQ:3}, against our initial contradictory assumption.

Having established~\eqref{UQ:3}, we can now complete the proof
of Theorem~\ref{UQ:1}, since, if~$A$ is Lipschitz, we can use~\eqref{UQ:3}
and the classical unique continuation principle in~\cite{MR882069}
and obtain~\eqref{UQ:4}, as desired.

\section{Proof of Theorem~\ref{BLOW-a}}\label{senrtoesyuiert03}

By~\eqref{MAIN EQ} and~\eqref{ulam}, we see that, if~$x\in\Omega$ and
$\lambda$ is sufficiently small,
\begin{equation}\label{Eq:la:1}
\begin{split}
0\,&= {\rm div}\,\big(A(\lambda x)\,\nabla u_\lambda (x)\big)-
\frac{\lambda^2}{\sqrt{H(\lambda)}}\,g(\lambda x,\sqrt{H(\lambda)}\,u_\lambda (x))
\\ &= {\rm div}\,\big(A_\lambda (x)\,\nabla u_\lambda (x)\big)-
g_\lambda (x,u_\lambda (x)),
\end{split}\end{equation}
where
\begin{eqnarray*}
A_\lambda (x)&:=&A(\lambda x)\\
{\mbox{and}}\qquad g_\lambda (x,t)&:=&
\frac{\lambda^2}{\sqrt{H(\lambda)}}\,g(\lambda x,\sqrt{H(\lambda)}\,t).
\end{eqnarray*}
Similarly, we see that, if~$x\in\partial\Omega$,
\begin{equation}\label{Eq:la:2}
\begin{split}
0\,&= \frac{\lambda}{\sqrt{H(\lambda)}}\,\Big(
A(\lambda x)\nabla u(\lambda x)\cdot \nu(\lambda x)-f(\lambda x,u(\lambda x))\Big)\\
&=
A_\lambda( x)\nabla u_\lambda( x)\cdot \nu(x)-
\frac{\lambda}{\sqrt{H(\lambda)}}\,f(\lambda x,\sqrt{H(\lambda)}\,u_\lambda (x))\\
&= 
A_\lambda( x)\nabla u_\lambda( x)\cdot \nu(x)-f_\lambda(x,u_\lambda(x)),
\end{split}\end{equation}
where
$$ f_\lambda(x,t):=
\frac{\lambda}{\sqrt{H(\lambda)}}\,f(\lambda x,\sqrt{H(\lambda)}\,t).$$
Now, in the notation of~\eqref{DeH DEF}, we write~$D_{u,A,f,g}$
and~$H_{u,A}$ to emphasize their dependences. In the same way,
in the notation of~\eqref{DeH DEN}, we write~${\mathcal{N}}_{u,A,f,g}$.
For short, we drop the indexes when they refer to the original
configuration in~\eqref{MAIN EQ} and we write
\begin{equation}\label{notazz}
D_\lambda:=D_{u_\lambda,A_\lambda,f_\lambda,g_\lambda},\qquad
H_\lambda:=
H_{u_\lambda,A_\lambda}
\qquad{\mbox{and}}\qquad
{\mathcal{N}}_\lambda:=
{\mathcal{N}}_{u_\lambda,A_\lambda,f_\lambda,g_\lambda}.\end{equation}
We remark that
\begin{eqnarray*}
H_\lambda(r)&=&
r^{1-n}\int_{\partial B_r\cap\Omega} A_\lambda(x)\,u^2_\lambda(x)
\,d{\mathcal{H}}^{n-1}_x\\&=&\frac{r^{1-n}}{H(\lambda)}
\,\int_{\partial B_r\cap\Omega} A(\lambda x)\,u^2(\lambda x)
\,d{\mathcal{H}}^{n-1}_x\\&=&
\frac{( \lambda r)^{1-n}}{H(\lambda)}
\,\int_{\partial B_{\lambda r}\cap\Omega} A(y)\,u^2(y)
\,d{\mathcal{H}}^{n-1}_y\\
&=&\frac{H(\lambda r)}{H(\lambda)}.
\end{eqnarray*}
In addition,
\begin{eqnarray*}
D_\lambda(r)&=&
r^{2-n}\int_{B_r\cap\Omega} A_\lambda(x)\,|\nabla u_\lambda(x)|^2\,dx
-r^{2-n}\int_{B_r\cap\partial\Omega} f_\lambda(x,u_\lambda(x))\,u_\lambda(x)\,d{\mathcal{H}}^{n-1}_x
\\&&\qquad\qquad
+r^{2-n}\int_{B_r\cap \Omega} g_\lambda(x,u_\lambda(x))\,u_\lambda(x)\,dx\\
&=&\frac{\lambda^2 r^{2-n}}{H(\lambda)}
\int_{B_r\cap\Omega} A(\lambda x)\,|\nabla u(\lambda x)|^2\,dx
-
\frac{\lambda r^{2-n}}{{H(\lambda)}}\,
\int_{B_r\cap\partial\Omega} 
f(\lambda x,u(\lambda x))
\,u(\lambda x)\,d{\mathcal{H}}^{n-1}_x
\\&&\qquad\qquad
+\frac{\lambda^2 r^{2-n}}{{H(\lambda)}}\,\int_{B_r\cap \Omega} 
g(\lambda x, u(\lambda x))\,u(\lambda x)\,dx\\
&=&\frac{(\lambda r)^{2-n}}{H(\lambda)}
\int_{B_{\lambda r}\cap\Omega} A(y)\,|\nabla u(y)|^2\,dy
-
\frac{(\lambda r)^{2-n}}{{H(\lambda)}}\,
\int_{B_{\lambda r}\cap\partial\Omega} 
f(y,u(y))
\,u(y)\,d{\mathcal{H}}^{n-1}_y
\\&&\qquad\qquad
+\frac{(\lambda r)^{2-n}}{{H(\lambda)}}\,\int_{B_{\lambda r}\cap \Omega} 
g(y, u(y))\,u(y)\,dy\\
&=&\frac{D(\lambda r)}{H(\lambda)},
\end{eqnarray*}
and therefore
\begin{equation}\label{78dh2eiyqwhdhffggfg}
{\mathcal{N}}_\lambda(r)=\frac{D(\lambda r)}{H(\lambda r)}=
{\mathcal{N}}(\lambda r).
\end{equation}
This and~\eqref{LIMITgamma} give that, for all $r>0$,
\begin{equation*}
\lim_{\lambda\searrow0}{\mathcal{N}}_\lambda(r)=
\gamma,
\end{equation*}
for some finite $\gamma\geq 0$.

Now we claim that, for all $R>0$ and $\lambda\in (0,r_0/R)$,
\begin{equation}\label{H1bound}
\| u_\lambda\|_{H^1(\Omega\cap B_R)}\le C_R,
\end{equation}
for some~$C_R>0$ (eventually depending on $R$). To this end, we
exploit~\eqref{DEE}, \eqref{Epci}, \eqref{eq:doub_gen},
and~\eqref{LAois} to see that, for all $\lambda\in(0,r_0/R)$,
\begin{equation}\label{9iuj52goq}
\begin{split}&
\int_{\Omega\cap B_R} A_\lambda (x)\,|\nabla u_\lambda (x)|^2\,dx=
\frac{\lambda^{2}}{H(\lambda)}
\int_{\Omega\cap B_R} A(\lambda x)\,|\nabla u(\lambda x)|^2\,dx\\&\qquad=
\frac{\lambda^{2-n}}{H(\lambda)}
\int_{\Omega\cap B_{R\lambda}} A(y)\,|\nabla u(y)|^2\,dy=
R^{n-2}\frac{E(\lambda R)}{H(\lambda)}\le 2R^{n-2}\,\frac{H(\lambda
  R)+D(\lambda R)}{H(\lambda)}\\&\qquad=
2R^{n-2}\,\frac{H(\lambda
  R)}{H(\lambda)}
(1+{\mathcal{N}}(\lambda R))\le C_R,\end{split}
\end{equation}
for some~$C_R>0$ depending on $R$.

Moreover, using again \eqref{eq:doub_gen}, we observe that
\begin{equation}\label{usia1}\begin{split}& \int_{\partial B_R\cap\Omega}
A_\lambda(x)\,u^2_\lambda(x)\,d{\mathcal{H}}^{n-1}_x=
\frac{1}{H(\lambda)}
\int_{\partial B_R\cap\Omega}
A(\lambda x)\,u^2(\lambda x)\,d{\mathcal{H}}^{n-1}_x\\&\qquad=
\frac{\lambda^{1-n}}{H(\lambda)}
\int_{\partial B_{R\lambda}\cap\Omega}
A(y)\,u^2(y)\,d{\mathcal{H}}^{n-1}_y=R^{n-1}\frac{H(R\lambda)}{H(\lambda)}\leq
C_R,
\end{split}\end{equation}
up to renaming $C_R$. Hence, recalling Corollary \ref{cor:dis} (used here with~$\mu:=0$, $r:=R$,
and on the function~$u_\lambda$ and with weight~$A_\lambda$)
and~\eqref{9iuj52goq},
\begin{equation}\label{8uhscUSU}
\begin{split}
\int_{\Omega\cap B_R} A_\lambda(x)\,u^2_\lambda(x)\,dx\,&
\leq C_R\,\left( \int_{\partial B_R\cap
  \Omega}A_\lambda(x)\,u_\lambda^2(x)\,d{\mathcal{H}}^{n-1}_x+
  \int_{\Omega\cap B_R} A_\lambda(x)|\nabla u_\lambda(x)|^2 \,dx\right)
\\ &\le C_R,
\end{split}\end{equation}
up to renaming~$C_R$. This inequality and~\eqref{9iuj52goq}, combined with~\eqref{MUCK},
give~\eqref{H1bound}, as desired.

Now, from \eqref{H1bound} and a diagonal process, we deduce that,
along  a subsequence,
$u_\lambda$ converges a.e. in~$\Omega$, strongly in~$L^2(\Omega\cap B_R)$
and weakly in~$H^1(\Omega\cap B_R)$ for all $R>0$, as $\lambda\searrow0$.
Consistently with the notation in Theorem~\ref{BLOW-a},
we denote by~$\tilde u$ this limit; we observe that $\tilde u\in \bigcap_{R>0}
H^1(\Omega\cap B_R)$.

As a particular case of \eqref{usia1} with $R=1$ we have that 
\[
 \int_{\partial B_1\cap\Omega}
A_\lambda(x)\,u^2_\lambda(x)\,d{\mathcal{H}}^{n-1}_x=1
\]
which, in view of the compactness of the trace embedding
$H^1(\Omega\cap B_1)\hookrightarrow \hookrightarrow L^2(\Omega\cap \partial B_1)$, implies that 
\begin{equation}\label{eq:tilde-u-NB}
 \int_{\partial B_1\cap\Omega}
\tilde u^2(x)\,d{\mathcal{H}}^{n-1}_x=\lim_{\lambda\searrow0} \int_{\partial B_1\cap\Omega}
A_\lambda(x)\,u^2_\lambda(x)\,d{\mathcal{H}}^{n-1}_x=1.
\end{equation}
Hence $\tilde u\not\equiv 0$.

We observe that, by~\eqref{g:grow}, for every~$x\in B_1$,
\begin{equation}\label{sti:gla}\begin{split}&
| g_\lambda (x,u_\lambda (x))|
=\frac{\lambda^2}{\sqrt{H(\lambda)}}\,\big|g(\lambda x,\sqrt{H(\lambda)}\,
u_\lambda (x))\big|\\
&\qquad\le
C\,\lambda^\delta\,A(\lambda x)\,|x|^{\delta-2}\,|u_\lambda( x)|\\
&\qquad\le
C\,\lambda^\delta\,|x|^{\delta-2}\,|u_\lambda( x)|
,\end{split}
\end{equation}
up to renaming~$C$ line after line.

Moreover, by~\eqref{STR:HY2},
\begin{equation}\label{sti:fla}
\begin{split}&
|f_\lambda(x,u_\lambda(x))|=
\frac{\lambda}{\sqrt{H(\lambda)}}\,
\big| f(\lambda x,\sqrt{H(\lambda)}\,u_\lambda(x))\big|\\&\qquad
\le C\,\lambda^\delta\,A(\lambda x)\,|x|^{\delta-1}\,|u_\lambda( x)|\\
&\qquad\le
C\,\lambda^\delta\,|x|^{\delta-1}\,|u_\lambda( x)|.
\end{split}
\end{equation}
Now we claim that, for all $R>0$,
\begin{equation}\label{djj02856ssd}
\begin{split}&
\lim_{\lambda\searrow0} \int_{\Omega\cap B_R} g_\lambda(x,u_\lambda(x))\,u_\lambda(x)\,dx=0\\
{\mbox{and }}\quad&
\lim_{\lambda\searrow0} \int_{\partial\Omega\cap B_R} f_\lambda(x,u_\lambda(x))\,u_\lambda(x)\,d{\mathcal{H}}^{n-1}_x=0.
\end{split}
\end{equation}
Indeed, using~\eqref{sti:gla}, Corollary \ref{cor:dis} (used here with~$A:=1$,
$r:=R$, and~$\mu:=2-\delta$), and \eqref{MUCK}, we see that
\begin{eqnarray*}
&& \left|\int_{\Omega\cap B_R} g_\lambda(x,u_\lambda(x))\,u_\lambda(x)\,dx\right|\le
C\,\lambda^\delta\,\int_{\Omega\cap B_R}|x|^{\delta-2}\,u^2_\lambda( x)\,dx\\&&\qquad\le
C_R\,\lambda^\delta\,\left( 
\int_{\partial B_R\cap\Omega} u^2_\lambda( x)\,d{\mathcal{H}}^{n-1}_x
+
\int_{\Omega\cap B_R}\,|\nabla u_\lambda( x)|^2\,dx
\right)
\\&&\qquad\le
C_R\,\lambda^\delta\,\left( 
\int_{\partial B_R\cap \Omega} A_\lambda(x)\,u^2_\lambda( x)\,d{\mathcal{H}}^{n-1}_x
+
\int_{\Omega\cap B_R}A_\lambda(x)\,|\nabla u_\lambda( x)|^2\,dx
\right).
\end{eqnarray*}
{F}rom this, \eqref{usia1} and~\eqref{9iuj52goq}, we deduce that
\begin{equation*}
\left|\int_{\Omega\cap B_1} g_\lambda(x,u_\lambda(x))\,u_\lambda(x)\,dx\right|\le
C_R\lambda^\delta,
\end{equation*}
up to renaming~$C_R>0$.

This proves the first claim in~\eqref{djj02856ssd}, and we now prove
the second. For this, using~\eqref{sti:fla}, and then Lemma~\ref{TRACEIN1}
(with~$A:=1$, $r:=R$ and~$\gamma:=1-\delta$)
we find that
\begin{eqnarray*}
&&\left|\int_{\partial\Omega\cap B_R}
f_\lambda(x,u_\lambda(x))\,u_\lambda(x)\,d{\mathcal{H}}^{n-1}_x
\right|\le
C\,\lambda^\delta\,
\int_{\partial\Omega\cap B_R} |x|^{\delta-1}\,
u^2_\lambda( x)\,d{\mathcal{H}}^{n-1}_x\\
&&\qquad\le C_R\,\lambda^\delta\,
\int_{\Omega\cap B_R}\left(|\nabla u_\lambda(x)|^2
+\frac{u^2_\lambda(x)}{|x|^{2-\delta}}\right)\,dx.
\end{eqnarray*}
Hence, using Corollary \ref{cor:dis} as before, we obtain that
$$ \left|\int_{\partial\Omega\cap B_R}
f_\lambda(x,u_\lambda(x))\,u_\lambda(x)\,d{\mathcal{H}}^{n-1}_x
\right|\le C_R\,\lambda^\delta,
$$
which implies the second claim in~\eqref{djj02856ssd}.
This completes the proof of~\eqref{djj02856ssd}.

Now we claim that
\begin{equation}\label{eqj-tilsa}
\begin{cases}
\Delta \tilde u=0 & {\mbox{ in }}\Omega,\\
\displaystyle\frac{\partial\tilde u}{\partial\nu}=0& {\mbox{ on }}\partial\Omega.
\end{cases}
\end{equation}
To this end, we exploit~\eqref{Eq:la:1} and~\eqref{Eq:la:2} and, given~$
\varphi\in  C^\infty_0(\R^n)$, we write that
\begin{eqnarray*}
0&=& \int_{\Omega} {\rm div}\,\Big(
A_\lambda(x)\,\nabla u_\lambda(x)\Big)\varphi(x)\,dx-
\int_{\Omega} g_\lambda(x,u_\lambda(x))\,\varphi(x)\,dx\\
&=& \int_{\partial\Omega} 
A_\lambda(x)\,\varphi(x)\nabla u_\lambda(x)\cdot \nu(x)\,d{\mathcal{H}}^{n-1}_x
\\&&\qquad - \int_{\Omega} 
A_\lambda(x)\,\nabla u_\lambda(x)\cdot\nabla\varphi(x)\,dx
-\int_{\Omega} g_\lambda(x,u_\lambda(x))\,\varphi(x)\,dx\\
&=& \int_{\partial\Omega} f_\lambda(x,u_\lambda(x))\,
\varphi(x)\,d{\mathcal{H}}^{n-1}_x
\\&&\qquad - \int_{\Omega} 
A_\lambda(x)\,\nabla u_\lambda(x)\cdot\nabla\varphi(x)\,dx
-\int_{\Omega} g_\lambda(x,u_\lambda(x))\,\varphi(x)\,dx.
\end{eqnarray*}
Hence, in light of~\eqref{ACONT}, \eqref{sti:gla} and~\eqref{sti:fla},
\begin{eqnarray*}
&& \left|\int_{\Omega} 
\nabla \tilde u(x)\cdot\nabla\varphi(x)\,dx\right|=\lim_{\lambda\searrow0}\left|
\int_{\Omega} 
A_\lambda(x)\,\nabla u_\lambda(x)\cdot\nabla\varphi(x)\,dx\right|\\
&&\qquad=\lim_{\lambda\searrow0}\left|
\int_{\partial\Omega} f_\lambda(x,u_\lambda(x))\,
\varphi(x)\,d{\mathcal{H}}^{n-1}_x
-\int_{\Omega} g_\lambda(x,u_\lambda(x))\,\varphi(x)\,dx\right|\\&&\qquad
\le C\,\lim_{\lambda\searrow0}\lambda^\delta\left(
\int_{\partial\Omega} |x|^{\delta-1}\,|u_\lambda(x)|\,|
\varphi(x)|\,d{\mathcal{H}}^{n-1}_x
+
\int_{\Omega} |x|^{\delta-2}\,|u_\lambda(x)|\,|\varphi(x)|\,dx
\right)\\&&\qquad
\le C\,\lim_{\lambda\searrow0}\lambda^\delta\left(
\int_{\partial\Omega\cap B_R} |x|^{\delta-1}\,|u_\lambda(x)|^2\,d{\mathcal{H}}^{n-1}_x+
\int_{\partial\Omega \cap B_R} |x|^{\delta-1}\,|
\varphi(x)|^2\,d{\mathcal{H}}^{n-1}_x
\right.\\
&&\qquad\left.\qquad\qquad
+
\int_{\Omega \cap B_R} |x|^{\delta-2}\,|u_\lambda(x)|^2\,dx+
\int_{\Omega\cap B_R} |x|^{\delta-2}\,|\varphi(x)|^2\,dx
\right)\\&&\qquad\le C'\,\lim_{\lambda\searrow0}\lambda^\delta\left(1+
\int_{\partial\Omega\cap B_R} |x|^{\delta-1}\,|u_\lambda(x)|^2\,d{\mathcal{H}}^{n-1}_x+
\int_{\Omega\cap B_R} |x|^{\delta-2}\,|u_\lambda(x)|^2\,dx
\right),
\end{eqnarray*}
where~$C'$, $R>0$ may also depend on~$\varphi$.
Consequently, using
Corollary \ref{cor:dis} and Lemma~\ref{TRACEIN1} as before, we obtain
$$ \left|\int_{\Omega} 
\nabla \tilde u(x)\cdot\nabla\varphi(x)\,dx\right|\le
C'\,\lim_{\lambda\searrow0}\lambda^\delta,$$
up to renaming~$C'>0$, that is
$$ \int_{\Omega} 
\nabla \tilde u(x)\cdot\nabla\varphi(x)\,dx=0.$$
Since this identity holds true for all~$\varphi\in C^\infty_0(\R^n)$,
we have completed the proof of~\eqref{eqj-tilsa}.

We now show that
\begin{equation}\label{H2bound}
{\mbox{$u_\lambda$ converges strongly to $\tilde u$ in~$H^1(\Omega\cap B_1)$,
as $\lambda\searrow0$.}}
\end{equation}
for some~$C>0$. To accomplish this, we will exploit elliptic regularity theory,
see e.g. Theorem~8.13 in~\cite{MR2399851}
(with the notation in Example 6.2 on page~314 in~\cite{MR2399851}
for the definition of the norms) or~\cite{MR0125307, MR0162050} and Theorem~5.1
in~\cite{MR0350177},
considering a set~$\Omega_1$ with smooth boundary
and such that~$\Sigma\subset \Omega_1\subset \Omega\cap (B_2\setminus B_{1/2})$.
In this way, by~\eqref{Eq:la:1} and~\eqref{Eq:la:2},
\begin{equation}\label{SA}
\| u_\lambda\|_{H^2(\Omega_1)}\le C\,\Big( 1+
\|u_\lambda\|_{L^2(\Omega_1)}+
\|g_\lambda(\cdot,u_\lambda)\|_{L^2(\Omega\cap(B_2\setminus B_{1/2}))}+
\|f_\lambda(\cdot,u_\lambda)\|_{H^{1/2}((\partial\Omega)\cap(B_2\setminus B_{1/2}))}
\Big).
\end{equation}
Moreover, in light of~\eqref{8uhscUSU} and~\eqref{sti:gla},
\begin{equation}\label{8wuief65748399312iofv}
\begin{split}
\|g_\lambda(\cdot,u_\lambda)\|_{L^2(\Omega\cap(B_2\setminus B_{1/2}))}^2\,&=
\int_{ \Omega\cap(B_2\setminus B_{1/2})} | g_\lambda (x,u_\lambda (x))|^2\,dx\\
&\le C\,\lambda^{2\delta}\,\int_{ \Omega\cap(B_2\setminus B_{1/2})}|x|^{2(\delta-2)}\,
|u_\lambda( x)|^2\,dx\\
&\le C\,\lambda^{2\delta}\,\int_{ \Omega\cap B_2 }
|u_\lambda( x)|^2\,dx
\\&\le C\,\lambda^{2\delta}.
\end{split}
\end{equation}
Similarly, recalling~\eqref{sti:fla} and \eqref{8uhscUSU},
\begin{eqnarray*}
\|f_\lambda(\cdot,u_\lambda)\|^2_{L^2(\Omega\cap(B_2\setminus B_{1/2}))}&=&
\int_{\Omega\cap(B_2\setminus B_{1/2})}
|f_\lambda(x,u_\lambda(x))|^2\,d x\\
&\le& C\,\lambda^{2\delta}\,
\int_{\Omega\cap(B_2\setminus B_{1/2})} |x|^{2(\delta-1)}
|u_\lambda(x)|^2\,d x\\
&\le& C\,\lambda^{2\delta}\,
\int_{\Omega\cap B_2}
 |u_\lambda(x)|^2\,dx\\
&\le& C\,\lambda^{2\delta}.
\end{eqnarray*}
Furthermore, from \eqref{STR:HY2DER}, \eqref{eq:f_t}, and \eqref{H1bound} it follows that
\begin{align*}
 & \|\nabla(f_\lambda(\cdot,u_\lambda))\|^2_{L^2(\Omega\cap(B_2\setminus
    B_{1/2}))}\\
&=\int_{\Omega\cap(B_2\setminus B_{1/2})}
\left|
\frac{\lambda}{\sqrt{H(\lambda)}}\left(\lambda\nabla_xf\left(\lambda
  x,\sqrt{H(\lambda)} u_\lambda(x)\right)+f_t\left(\lambda
  x,\sqrt{H(\lambda)} u_\lambda(x)\right) \sqrt{H(\lambda)}\nabla
  u_\lambda(x)\right)\right|^2dx\\
&\leq C \lambda^{2\delta}\int_{\Omega\cap
  B_2}(|u_\lambda(x)|^2+|\nabla u_\lambda(x)|^2)\,dx\\
&\leq C \lambda^{2\delta}
\end{align*}
Therefore
\[
\|f_\lambda(\cdot,u_\lambda)\|_{H^1(\Omega\cap(B_2\setminus B_{1/2}))}
\leq C\lambda^\delta
\]
which, in view of  the continuous trace embedding $H^1(\Omega\cap(B_2\setminus
B_{1/2}))\hookrightarrow H^{1/2}(\partial\Omega\cap(B_2\setminus
B_{1/2}))$,  yields 
\[
\|f_\lambda(\cdot,u_\lambda)\|_{H^{1/2}(\partial \Omega\cap(B_2\setminus B_{1/2}))}
\leq C\lambda^\delta
\]
up to renaming~$C$.
{F}rom this, \eqref{8uhscUSU},
\eqref{8wuief65748399312iofv} and~\eqref{SA}, we conclude that
$$ \| u_\lambda\|_{H^2(\Omega_1)}\le C,$$
again up to renaming~$C>0$. 
Thus, using the trace embedding,
$$ \| u_\lambda\|_{H^{3/2}(\Sigma)}\le C,$$
up to renaming~$C>0$, and consequently, up to a subsequence,
we obtain that
\begin{equation}\label{h1ka}
{\mbox{$u_\lambda$ converges to~$\tilde u$ in~$H^1(\Sigma)$.}}
\end{equation}
Now we notice that, exploiting~\eqref{Eq:la:1} and~\eqref{Eq:la:2},
\begin{eqnarray*}
0&=& \int_{\Omega\cap B_1} {\rm div}\,\Big(
A_\lambda(x)\,\nabla u_\lambda(x)\Big)u_\lambda(x)\,dx-
\int_{\Omega\cap B_1} g_\lambda(x,u_\lambda(x))\,u_\lambda(x)\,dx\\
&=& \int_{\partial(\Omega\cap B_1)} 
A_\lambda(x)\,u_\lambda(x)\nabla u_\lambda(x)\cdot \nu(x)\,d{\mathcal{H}}^{n-1}_x
\\&&\qquad - \int_{\Omega\cap B_1} 
A_\lambda(x)\,|\nabla u_\lambda(x)|^2\,dx
-\int_{\Omega\cap B_1} g_\lambda(x,u_\lambda(x))\,u_\lambda(x)\,dx\\
&=& \int_{\Sigma} 
A_\lambda(x)\,u_\lambda(x)\nabla u_\lambda(x)\cdot
\nu(x)\,d{\mathcal{H}}^{n-1}_x
+\int_{\partial\Omega\cap B_1} f_\lambda(x,u_\lambda(x))\,
u_\lambda(x)\,d{\mathcal{H}}^{n-1}_x
\\&&\qquad - \int_{\Omega\cap B_1} 
A_\lambda(x)\,|\nabla u_\lambda(x)|^2\,dx
-\int_{\Omega\cap B_1} g_\lambda(x,u_\lambda(x))\,u_\lambda(x)\,dx.
\end{eqnarray*}
Using this, \eqref{ACONT}, \eqref{djj02856ssd}
and~\eqref{h1ka}, we conclude that
\begin{eqnarray*}&&
\lim_{\lambda\searrow0}\int_{\Omega\cap B_1} |\nabla u_\lambda(x)|^2\,dx\\
&=&
\lim_{\lambda\searrow0}\int_{\Omega\cap B_1} 
A_\lambda(x)\,|\nabla u_\lambda(x)|^2\,dx\\
&=&
\lim_{\lambda\searrow0}
\int_{\Sigma} 
A_\lambda(x)\,u_\lambda(x)\nabla u_\lambda(x)\cdot
\nu(x)\,d{\mathcal{H}}^{n-1}_x
+\int_{\partial\Omega\cap B_1} f_\lambda(x,u_\lambda(x))\,
u_\lambda(x)\,d{\mathcal{H}}^{n-1}_x
\\&&\qquad 
-\int_{\Omega\cap B_1} g_\lambda(x,u_\lambda(x))\,u_\lambda(x)\,dx\\
&=&
\lim_{\lambda\searrow0}\int_{\Sigma} 
A_\lambda(x)\,u_\lambda(x)\nabla u_\lambda(x)\cdot
\nu(x)\,d{\mathcal{H}}^{n-1}_x\\&=&
\int_{\Sigma} 
\tilde u(x)\nabla \tilde u(x)\cdot
\nu(x)\,d{\mathcal{H}}^{n-1}_x.
\end{eqnarray*}
Hence, recalling~\eqref{eqj-tilsa},
\begin{eqnarray*}
\lim_{\lambda\searrow0}\int_{\Omega\cap B_1} |\nabla u_\lambda(x)|^2\,dx
&=&
\int_{\partial(\Omega\cap B_1)} 
\tilde u(x)\nabla \tilde u(x)\cdot
\nu(x)\,d{\mathcal{H}}^{n-1}_x\\
&=&
\int_{\Omega\cap B_1} {\rm div}\,\Big(
\tilde u(x)\nabla \tilde u(x)
\Big)\,dx\\&=&
\int_{\Omega\cap B_1} |\nabla \tilde u(x)|^2\,dx.
\end{eqnarray*}
Since the weak convergence and the convergence of the norm
imply the strong convergence in~$L^2(\Omega\cap B_1)$, we thereby
conclude that~$\nabla u_\lambda$ converges to~$\nabla\tilde u$ strongly
in~$L^2(\Omega\cap B_1,\R^n)$, and this gives~\eqref{H2bound},
as desired.

{F}rom~\eqref{H2bound} and \eqref{djj02856ssd},
 recalling~\eqref{eqj-tilsa} and the notation in~\eqref{notazz},
we conclude that
$$ \lim_{\lambda\searrow0} {\mathcal{N}}_\lambda(r)=
{\mathcal{N}}_{\tilde u,1,0,0}(r).$$
As a consequence, exploiting~\eqref{78dh2eiyqwhdhffggfg},
\begin{equation}\label{NMa-10} {\mathcal{N}}_{\tilde u,1,0,0}(r)=\gamma.\end{equation}
{F}rom this, we conclude that
\begin{equation}\label{HO1}
{\mbox{$\tilde u$ is positively homogeneous
of degree~$\gamma$,}}\end{equation} and hence we can write~$\tilde u$
as in~\eqref{tildeupsi}.

For completeness, we give a self-contained proof of~\eqref{HO1}
by arguing as follows. 
By~\eqref{NMa-10}, we know that~${\mathcal{N}}_{\tilde u,1,0,0}'(r)=0$,
and therefore, by~\eqref{DeH DEN},
$$ D'_{\tilde u,1,0,0}(r)\,H_{\tilde u,1}(r)-H'_{\tilde u,1}(r)
D_{\tilde u,1,0,0}(r)=0\quad\text{for all }r>0.$$
Hence, exploiting~\eqref{PLUAQ} in this setting, and recalling~\eqref{0102303},
we see that
\begin{eqnarray*}
0&=& r^{-n} H_{\tilde u,1}(r)\;
\int_{\partial B_r\cap\Omega}
(\nabla \tilde u\cdot x)^2
-r^{3-2n}\left(\int_{\partial B_r\cap\Omega}
\tilde u\nabla \tilde u\cdot\nu\right)^2 \\
&=&
r^{3-2n}\left[\int_{\partial B_r\cap\Omega} \tilde u^2
\int_{\partial B_r\cap\Omega}
(\nabla \tilde u\cdot \nu)^2
-\left(\int_{\partial B_r\cap\Omega}
\tilde u\nabla \tilde u\cdot\nu\right)^2\right]\quad\text{for all }r>0.
\end{eqnarray*}
By the Cauchy-Schwarz Inequality, the latter term is nonnegative,
and consequently we find that~$\tilde u$ is proportional to~$\nabla \tilde u\cdot \nu$.
Accordingly, we have that~$\tilde u$ is a positively
homogeneous function, of some degree~$\gamma'$.

Then, using~\eqref{NMa-10} once again
\begin{equation}\label{860329eufj0987kap}
\begin{split}
& \gamma\, \int_{\partial B_1\cap\Omega} \tilde u(x)\,\frac{\partial\tilde u}{\partial\nu}(x)
\,d{\mathcal{H}}^{n-1}_x
= \gamma\, {\gamma'}\,\int_{\partial B_1\cap\Omega} \tilde u^2(x)
\,d{\mathcal{H}}^{n-1}_x=
\gamma \,{\gamma'}\,H_{\tilde u,1}(1)\\&\qquad=\gamma'D_{\tilde u,1,0,0}(1)
= {\gamma'}\,\int_{B_1\cap\Omega} |\nabla \tilde u(x)|^2\,dx.
\end{split}\end{equation}
On the other hand, by~\eqref{eqj-tilsa},
$$ 
\int_{\partial B_1\cap\Omega}
\tilde u(x)\,\frac{\partial\tilde u}{\partial\nu}(x)
\,d{\mathcal{H}}^{n-1}_x=
\int_{\partial(B_1\cap\Omega)}
\tilde u(x)\,\frac{\partial\tilde u}{\partial\nu}(x)
\,d{\mathcal{H}}^{n-1}_x=
\int_{B_1\cap\Omega} |\nabla \tilde u(x)|^2\,dx
.$$
Plugging this information into~\eqref{860329eufj0987kap}, we thereby conclude that
$$ \gamma\, \int_{B_1\cap\Omega} |\nabla \tilde u(x)|^2\,dx
= {\gamma'}\,\int_{B_1\cap\Omega} |\nabla \tilde u(x)|^2\,dx,$$
and then~$\gamma'=\gamma$.
This completes the proof of~\eqref{HO1} (and thus of~\eqref{tildeupsi}).

We also remark that, by~\eqref{tildeupsi} and~\eqref{eqj-tilsa},
using the notation~$\rho:=|x|$ and~$\vartheta:=x/|x|$,
$$ 0=\Delta \tilde u(x)= \gamma(\gamma-1)\rho^{\gamma-2}\psi(\vartheta)
+(n-1)\gamma\rho^{\gamma-2}\psi(\vartheta)+\rho^{\gamma-2}\Delta_{S^{n-1}}\psi(\vartheta),
$$
and therefore~$\psi$ is an eigenfunction of te operator $\mathcal
L_\Sigma$; the Neumann boundary condition of~$\psi$ also follows from
the one of~$\tilde u$ in~\eqref{eqj-tilsa}.

Furthermore, by \eqref{eq:tilde-u-NB} and \eqref{tildeupsi} 
\begin{eqnarray*}
1&=&\int_{\partial B_1\cap\Omega}
\tilde u^2(x)\,d{\mathcal{H}}^{n-1}_x\\
&=& \int_{\partial B_1\cap\Omega} |x|^{2\gamma}
\psi^2\left(\frac{x}{|x|}\right)\,d{\mathcal{H}}^{n-1}_x\\&=&
\int_{\partial B_1\cap\Omega}
\psi^2(x)\,d{\mathcal{H}}^{n-1}_x,
\end{eqnarray*}
which gives~\eqref{NORMALIZ}. The proof of Theorem~\ref{BLOW-a}
is thereby complete.

\section{Proof of Theorem~\ref{UnA-ass}}\label{oedkcpoetrtegfi}

First, we prove~\eqref{UQ:3X}. We argue by contradiction,
supposing that~\eqref{UQ:3X} does not hold, and therefore~\eqref{NONTRIVIAL}
is satisfied. Hence, we are in the position of using Theorem~\ref{BLOW-a},
and we let~$\tilde u$ and~$\psi$ as in~\eqref{tildeupsi}.  We note
that, by \eqref{eqj-tilsa} and elliptic regularity theory, we have
that $\tilde u$ is smooth on $\overline\Omega\setminus\{0\}$.

We observe that the trace of $\tilde u$ on $B_1\cap \partial\Omega$
(which belongs to $L^2(B_1\cap \partial\Omega)$ by trace embeddings)
cannot vanish identically, i.e. 
\begin{equation}\label{eq:nottriv}
\tilde u\not\equiv0\quad\text{on
$B_1\cap \partial\Omega$},
\end{equation}
 otherwise $\tilde u$ would be
a harmonic function with homogeneous Dirichlet and Neumann conditions
on $B_1\cap \partial\Omega$,
and then necessarily~$\tilde u$ would vanish identically in $B_1\cap
\Omega$ (otherwise its trivial extension would violate classical
unique continuation principles),
in contradiction with~\eqref{NORMALIZ}.


From assumption \eqref{UQ:2X} it follows that, for all $k\in\N$ 
\begin{equation}\label{eq:6}
\lambda^{-k}u(\lambda\cdot)\to 0\quad\text{in
}L^2(B_1\cap \partial\Omega).
\end{equation}
Since, in view of \eqref{ulam},
\begin{align*}
  \frac{\sqrt{H(\lambda})}{\lambda^k}=\frac{\|\lambda^{-k}u(\lambda\cdot)\|_{L^2(B_1\cap \partial\Omega)}}{\|u_\lambda\|_{L^2(B_1\cap \partial\Omega)}}
\end{align*}
and, by Theorem \ref{BLOW-a}, $u_\lambda \to \tilde u$ in
$L^2(B_1\cap \partial\Omega)$ along a subsequence, from
\eqref{eq:nottriv} and \eqref{eq:6} we conclude that
\begin{equation*}
\lim_{\lambda\searrow0}\frac{\sqrt{H(\lambda)}}{\lambda^k}=0,
\end{equation*}
for all~$k\in\N$.
Consequently, for all~$k\in\N$, there exists~$\lambda_0(k)\in(0,r_0/2)$
such that, for all~$\lambda\in(0,\lambda_0(k)]$,
\begin{equation}\label{8ywishcjx bc0293ryeu}
\frac{\sqrt{H(\lambda)}}{\lambda^k}\le1.
\end{equation}
On the other hand, by~\eqref{DOU}, 
$$ H(2^m \lambda)\le C^m H(\lambda)$$
for all~$m\in\N$ and $\lambda\in (0, 2^{-m}r_0)$, for a suitable~$C>0$
independent of $\lambda$ and $m$. 
This and~\eqref{8ywishcjx bc0293ryeu} give that,
for all~$k$, $m\in\N$ and
for all~$\lambda\in(0,\min\{\lambda_0(k),2^{-m}r_0\})$,
\begin{equation*}
H(2^m \lambda)\le C^m\lambda^{2k}.
\end{equation*}
As a consequence, recalling~\eqref{DeH DEF} and integrating,
\begin{eqnarray*}&&
\frac{1}{2^{m}}
\int_{ B_{2^m\lambda}\cap\Omega} A(x)\,u^2(x)
\,dx
=
\frac{1}{2^{m}}
\int_0^{2^m\lambda}\left[
\int_{\partial B_{\rho}\cap\Omega} A(x)\,u^2(x)
\,d{\mathcal{H}}^{n-1}_x\right]\,d\rho\\&&\qquad=
\int_0^{\lambda}\left[
\int_{\partial B_{2^m r}\cap\Omega} A(x)\,u^2(x)
\,d{\mathcal{H}}^{n-1}_x\right]\,dr
=\int_0^\lambda 
(2^m r)^{n-1}
H(2^m r)\,dr\\&&\qquad
\le 2^{m(n-1)}C^m \int_0^\lambda r^{n-1+2k}\,dr
=\frac{2^{m(n-1)}C^m \,\lambda^{n+2k}}{n+2k},\end{eqnarray*}
for all~$k$, $m\in\N$ and
for all~$\lambda\in(0,\min\{\lambda_0(k),2^{-m}r_0\})$.

We  choose $m_\lambda\in\N$
such that
\begin{equation}\label{7-394} 
\left|\log_2\bigg(\frac{2\lambda}{r_0}\bigg)\right|\le m_\lambda<
 1+\left|\log_2\bigg(\frac{2\lambda}{r_0}\bigg)\right|,\end{equation}
so that $\lambda< 2^{-m_\lambda}r_0$ for all $\lambda<\frac{r_0}2$.
Then
we find that
$$ 
\int_{ B_{2^m\lambda}\cap\Omega} A(x)\,u^2(x)
\,dx\le \frac{2^{m_\lambda n}\;C^{m_\lambda} \,\lambda^{n+2k}}{n+2k},$$
for all~$k\in\N$ and
for all~$\lambda\in(0,\lambda_0(k)]$.

Hence, since, by~\eqref{7-394}, we know that~$2^{m_\lambda}\lambda\in\left[\frac{r_0}2,r_0\right]$,
$$ 
\int_{ B_{\frac{r_0}{2}}\cap\Omega} A(x)\,u^2(x)
\,dx\le \frac{(2^{n}\,C)^{1+|\log_2\frac{2\lambda}{r_0}|} \,\lambda^{n+2k}}{n+2k}\le
\frac{(2^{n}\,C)^{-2 \log_2\frac{2\lambda}{r_0}} \,\lambda^{n+2k}}{n+2k}=\kappa\frac{\lambda^{n+2k-\theta}}{n+2k},
$$
for some suitable~$\theta,\kappa>0$ depending only on $n,C,r_0$ (but
independent of $k$),
for all~$k\in\N$ and
for all~$\lambda\in(0,\min\{\lambda_0(k),r_0/4\})$.

Accordingly, choosing~$k\in\N$ sufficiently large
such that~$n+2k-\theta>0$ and sending~$\lambda\searrow0$, we conclude that
$$ \int_{B_{\frac{r_0}{2}}\cap\Omega} A(x)\,u^2(x)
\,dx=0.$$
This gives that~\eqref{UQ:3X} holds true, in contradiction with
our initial hypothesis.

This completes the proof of~\eqref{UQ:3X}.
Finally, the proof of~\eqref{UQ:4X}
is identical to the proof of~\eqref{UQ:4}, hence the proof of
Theorem~\ref{UnA-ass} is complete.

\vfill

{\footnotesize

\noindent {\em Addresses:} \\

\noindent {\sc
Serena Dipierro.}
Department of Mathematics
and Statistics,
University of Western Australia,
35 Stirling Hwy, Crawley WA 6009, Australia.\medskip

\noindent {\sc
Veronica Felli.}
Dipartimento di Scienza dei Materiali,
Universit\`a di Milano-Bicocca,
Via Cozzi 55, 20125 Milano, Italy.\medskip

\noindent
{\sc Enrico Valdinoci.}
Department of Mathematics
and Statistics,
University of Western Australia,
35 Stirling Hwy, Crawley WA 6009, Australia.\medskip
\medskip

\noindent
E-mail: {\tt serena.dipierro@uwa.edu.au, veronica.felli@unimib.it,
enrico.valdinoci@uwa.edu.au}

}
\end{document}